\documentclass[11pt, a4paper]{article}
\usepackage{amsmath,amssymb,amsfonts}
\usepackage{amsthm}
\usepackage[greek,english]{babel}
\usepackage{verbatim}
\usepackage{a4wide}
\usepackage{epsfig,epic,eepic,graphicx}
\usepackage{graphicx}
\usepackage[active]{srcltx}
\usepackage{enumerate}

\newtheorem{theorem}{Theorem}
\newtheorem{lemma}[theorem]{Lemma}
\newtheorem{proposition}[theorem]{Proposition}

\newtheorem{definition}[theorem]{Definition}

\newtheorem{remark}[theorem]{Remark}
\newtheorem{corollary}[theorem]{Corollary}

\makeatletter

\@addtoreset{equation}{section}
\makeatother

\newcommand{\Om}{\Omega}
\newcommand{\be}{\begin{equation}}
\newcommand{\ee}{\end{equation}}
\newcommand{\R}{\mathbb R}
\newcommand{\N}{\mathbb N}
\newcommand{\pa}{\partial}
\newcommand{\na}{\nabla}
\newcommand{\dv}{\mathrm{div}\;}
\newcommand{\Supp}{\operatorname{Supp}}

\newcommand{\PV}{\operatorname{PV}}
\newcommand{\du}{\dot u_0}
\newcommand{\Td}{T_\delta}
\newcommand{\eps}{\epsilon}

\newcommand{\flap}{(-\Delta)^{1/2}}
\newcommand{\Si}{\Sigma}
\newcommand{\la}{\lambda}
\newcommand{\La}{\Lambda}
\newcommand{\ds}{\displaystyle}

\def\div{\hbox{div  }}
\title{On shape optimization problems \\ involving  the fractional laplacian} 
\author{\footnote{DMA/CNRS, Ecole Normale Sup\'erieure, 45 rue d'Ulm, 75005
   Paris, FRANCE, tel: +33 1 44 32 20 58, fax: +33 1 44 32 20 80.} Anne-Laure Dalibard 
and 
 \footnote{IMJ and University Paris 7, 175 rue du Chevaleret, 75013 Paris. } David G\'erard-Varet }
\date{}

\begin{document}
\bibliographystyle{amsplain}
\maketitle

\begin{abstract}
Our concern is  the computation of optimal shapes in problems involving $\flap$. We focus on the energy $J(\Omega)$ associated to the solution $u_\Omega$ of the basic Dirichlet problem $\flap u_\Omega = 1$ in $\Omega$, $ u = 0$ in  $\Omega^c$. We show that regular minimizers $\Omega$ of this energy  under a volume constraint are disks. Our proof goes through the explicit computation of the shape derivative (that seems to be completely new in the fractional context), and a refined adaptation of the moving plane method. 
  
\end{abstract}

 
\section{Introduction}

This article is concerned with shape optimization problems involving the fractional laplacian. 
The typical example we have in mind comes from  the following system:  
\be\label{EL}
\begin{aligned}
(-\Delta)^{1/2} u=1 \text{ on } \Om,\\
u=0\text{ on }\Om^c,
\end{aligned}
\ee
set for a bounded open set $\Omega$ of $\R^2$.  We wish to find the minimizers of  the associated energy  
\be\label{def:J}
J(\Om):= \inf_{ \substack{v\in H^{1/2}(\R^2),\\ v|_{\Omega^c}=0}}\left(\frac{1}{2}\left\langle (-\Delta)^{1/2}v,v\right\rangle_{H^{-1/2},H^{1/2}} - \int_{\R^2}v\right).
\ee
among  open sets $\Omega$ with prescribed measure (and a smoothness assumption). Beyond this specific example,  we wish to develop mathematical tools for  shape optimization in the context of fractional operators. 

\medskip
 Our original motivation comes from a drag reduction problem in microfluidics. Recent experiments, carried on liquids in microchannels, have suggested that drag is substantially lowered when the wall of the channel is water-repellent and rough \cite{Lauga,Vinogradova1}. The idea is 
 that the liquid sticks to the bumps of the roughness, but may slip over its humps, allowing for less friction at the boundary. Mathematically, one can consider Stokes equations for the liquid 
 (variable $(x,z) = (x_1, x_2, z)$, velocity field $u = (v,w) = (v_1,v_2,w)$):
 \begin{equation} \label{Stokes}
 -\Delta u + \na p = 0, \quad \div u = 0, \quad z > 0, 
\end{equation}
set above a flat surface $\{ z = 0 \}$. This flat surface replaces the rough hydrophobic one, and is composed of areas on which the fluid satisfies alternately perfect slip and no-slip boundary conditions. In the simplest models, these areas of perfect slip and no-slip form a periodic pattern, corresponding to a periodic pattern of humps and bumps. That means that the impermeability condition $w=0$ at $\{z= 0\}$ is completed with mixed Dirichlet/Navier conditions: 
\begin{equation} \label{DirNav}
 \pa_z v = 0 \:  \mbox{ in } \Omega, \quad \: v = 0   \mbox{ in } \Omega^c 
\end{equation}
 where $\Om\subset \R^2$ corresponds to the zones of perfect slip and $\Om^c$ to the zones of no-slip. The whole issue is to design $\Om$ so that the energy $J(\Om)$ associated with this problem is minimal for a fixed fraction of slip area. 
 Unfortunately, this optimization problem (Stokes operator, periodic pattern) is still out of reach. That is why we start with the simpler equations \eqref{EL}  (still difficult, and interesting on their own!). 
 Note that they can be seen as a scalar version of   \eqref{Stokes}-\eqref{DirNav}, replacing the Stokes operator by the Laplacian, and the Navier by the Neumann condition. Using the classical characterization of $\flap$ as the Dirichlet-to-Neumann operator leads to  system  \eqref{EL} and energy \eqref{def:J}. Again, we stress  that the  methods  developed in our paper may be useful in more elaborate contexts.

 \medskip
If the fractional laplacian in \eqref{EL}  is replaced by a standard laplacian (which leads to the classical Dirichlet energy), this problem is well-known and described in detail in the book \cite{HenrotPierre} by Henrot and Pierre (see also \cite{Serrin}). In this case, one can show that a smooth  domain $\Om$ minimizing the Dirichlet energy under the constraint $|\Om|=1$ is a disc. A standard proof of this result has two main steps:
\begin{enumerate}
 \item One computes the shape derivative associated with the Dirichlet energy for the laplacian. This leads to the following result: if $\Om$ is a minimizer of the Dirichlet energy and $u_\Om$ is the solution of the associated Euler-Lagrange equation, then $\pa_n u$ is constant on $\pa \Om$.

\item One analyzes an  overdetermined problem. More precisely, the idea is to prove that if there exists a function $u$ solving
$$
\begin{aligned}
-\Delta u=1\text{ in } \Om,\\
u=0\text{ on } \pa \Om,\\
\pa_n u=c\text{ on } \pa \Om,
\end{aligned}
$$
then all the connected components of $\Om$ are discs. This second step is achieved thanks to the moving plane method.
\end{enumerate}
Our goal in this article is to develop the same approach in the context of the fractional laplacian, showing radial symmetry of any smooth minimizer $\Omega$ of \eqref{def:J}. Accordingly, we start with the computation of the shape derivative.
\begin{theorem}
 Let $f\in \mathcal C^\infty(\R^2,\R)$, and let
$$
J_f(\Om):= \inf_{ \substack{v\in H^{1/2}(\R^2),\\ v|_{\Omega^c}=0}}\left(\frac{1}{2}\left\langle (-\Delta)^{1/2}v,v\right\rangle_{H^{-1/2},H^{1/2}} - \int_{\R^2}fv\right).
$$
Let $\zeta\in \mathcal C^\infty_0(\R^2, \R^2)$, and let $(\phi_t)_{t\in \R}$ be the flow associated with $\zeta$, namely
$
\dot \phi_t= \zeta(\phi_t),\quad \phi_0=\mathrm{Id}.
$

Let $\Om$ be an open set with $\mathcal C^\infty$ boundary, and let $u_{\Om,f}$ be the unique minimizer of $J_f(\Om)$, 
$$
\begin{aligned}
 (-\Delta)^{1/2} u_{\Om, f}=f\quad \text{in } \Om,\\
 u_{\Om, f}=0\quad \text{in } \Om^c.
\end{aligned}
$$

Then, denoting by $n(x)$ the outward pointing normal vector to $\pa \Om$,
\begin{enumerate}
 \item For all $x\in \pa \Om$, the limit
$$
\lim_{y\to x, y\in \Om}\frac{u_{\Om,f}(y)}{|(y-x)\cdot n(x)|^{1/2}}
$$
exists; we henceforth denote it by  $\pa_n^{1/2} u_{\Om,f}(x)$.

\item The function  $\pa_n^{1/2} u_{\Om,f}$ belongs to $\mathcal C^1(\pa \Om)$.

\item There exists an explicit constant $C_0$, which does not depend on $\Om$, such that
$$
\frac{d}{dt}J_f(\phi_t(\Om))_{|t=0}=C_0\int_{\pa \Om} (\pa_n^{1/2} u_{\Om,f})^2 \zeta\cdot n \:d\sigma.
$$
\end{enumerate}

\label{prop:der-forme-qcq}
\end{theorem}
This theorem implies easily 
 \begin{corollary}
Assume that there exists  an open set $\Om\subset \R^2$ such that $\Om$ has $\mathcal C^\infty$ regularity, $|\Om|=1$, and 
$$
J(\Om)= \inf_{\substack{\Om'\subset \R^2,\\|\Om'|=1}}J(\Om').
$$
Let $u_\Om$ be the solution of the associated Euler-Lagrange equation \eqref{EL}, and let $n(x)$ be the outward pointing normal at $x\in \pa \Om$. Then, $\pa_n^{1/2} u_\Om$ exists, and there exists a constant $c_0\geq 0$ such that
$$
\pa_n^{1/2} u_\Om(x)=c_0\quad\forall x\in \pa \Om.
$$

\label{thm:derivee-forme}
\end{corollary}
Taking into account this extra condition on the fractional normal derivative, we can then determine the minimizing domain:
\begin{theorem}
Let $\Om\subset \R^2$ be a $\mathcal C^\infty$ open set such that the system 
\be
\label{overdetermined}
\begin{aligned}
(-\Delta)^{1/2} u=1\text{ in }\Om,\\
u=0\text{ in } \Om^c,\\
 \pa_n^{1/2} u=c_0\text{ on } \pa \Om.
\end{aligned}
\ee
 has at least one solution. Assume that $\Om$ is connected. Then $\Om$ is a disc.
\label{thm:hyp-mobiles}

\end{theorem}

Of course, Theorem \ref{prop:der-forme-qcq} and Corollary \ref{thm:hyp-mobiles} imply immediately the following:
\begin{corollary}
 Let $\Om\subset \R^2$ be a $\mathcal C^\infty$ connected open set such that $|\Om|=1$, and 
$$
J(\Om)= \inf_{\substack{\Om' \subset \R^2,\\|\Om'|=1}}J(\Om').
$$
Then $\Om$ is a disc.
\label{cor:disc}

\end{corollary}

Let us make a few comments on our results. 
As regards Theorem  \ref{prop:der-forme-qcq}, we stress that, due to the non-locality of the fractional laplacian, the shape derivative of $J_f$ is hard to compute.  In the case of the classical laplacian, it is obtained through integration by parts, which are completely unavailable in the present context. The idea is to bypass the nonlocality by using an asymptotic expansion  of $u_{\Om,f}(x)$ as the distance between $x$ and  $\pa \Omega$ goes to zero. Such asymptotic expansion follows from  general results of \cite{CoDauDu}, on the solutions of linear elliptic systems in domains with cracks. We insist that  the proof of the theorem  neither uses  scalar arguments, nor Fourier-based calculations. In particular, we believe that its interet goes much beyond Corollary \ref{cor:disc} (that is finding the minimizer of  \eqref{def:J}). It is likely that it can be adapted to vectorial settings, or functionals with non-constant coefficients.
Notice also that Theorem \ref{prop:der-forme-qcq} and Corollary \ref{thm:derivee-forme}  do not require the connectedness of $\Omega$.  For our special case $f=1$, they imply  that the value of the fractional normal derivative is the same on all the boundaries of the connected components of $\Om$.

As regards Theorem \ref{thm:hyp-mobiles}, it is deduced from an adaptation of the moving plane method to our fractional setting. Again, this is not straightforward, as the standard  method relies heavily on the maximum principle and Hopf's Lemma, which are essentially local tools. To overcome our non-local problem,  we use  appropriate three-dimensional extensions of $u$, to which we can apply maximum principle methods in a classical context. Note however that we need to assume that $\Om$ is connected: this hypothesis is precisely due to the nonlocality of the fractional laplacian.

\medskip
We conclude this introduction by a brief review of related results. 
Let us first mention that the condition $\pa_n^{1/2} u=c_0$ appears in other problems related to the fractional laplacian. In \cite{CRS}, Caffarelli, Roquejoffre and Sire consider a minimization problem for another energy related to the fractional laplacian, and they prove that minimizers satisfy the condition $\pa_n^{1/2} u=c_0$ on $\pa \{u=0\}$. However, we emphasize that the issues of the present paper and those of \cite{CRS} are rather different. The problem addressed in \cite{CRS} is essentially a free-boundary problem (i.e. $\Om$ is not given, but is defined as $\{u>0\}$), and therefore  questions such as the regularity and the non-degeneracy of $u$, and the regularity of the free boundary, are highly non trivial and are at the core of the paper \cite{CRS}. Here, our goal is not to investigate these questions, but rather to derive information on the shape of $\Om$, assuming {\it a priori} regularity. As mentioned before, we rely on article \cite{CoDauDu} by Costabel and co-authors, which provides   asymptotic expansions for solutions of linear elliptic equations near cracks. As they derive accurate asymptotics, based on pseudo-differential calculus, they need  the domain to be $C^\infty$. In our context, only cruder information is needed (broadly, we need the first term in the expansion, and tangential regularity). Apart from these asymptotic expansions, the proofs of Theorem \ref{prop:der-forme-qcq} and Theorem \ref{thm:hyp-mobiles} use very little information on the regularity of $\pa\Om$ (existence of tangent and normal vectors, boundedness and regularity of the curvature). Hence, it is likely that our $C^\infty$ regularity requirement can be lowered.

As regards our adaptation of  the moving plane method, it relates to other results on the proof of radial symmetry for minimizers of nonlocal functionals: see for instance \cite{LuZh} on local Riesz potentials
$$ u(x) = \int_\Omega \frac{1}{|x-y|^{N-1}} $$
or  \cite{CaTa} on the radial symmetry of solutions of nonlinear equations  involving $A^{1/2}$, where  $A$ is the Dirichlet laplacian of a ball in $\R^n$. Note that in these two papers, ``local" maximum principles are still available, which helps. Further references (notably to article \cite{BiLoWa}) will be provided in due course.

Let us eventually point out that more direct proofs of the final Corollary \ref{cor:disc} might be available.
For instance, in the case of the classical laplacian, one way to proceed is to consider the auxiliary problem: {\em minimize}
$$  E(u)   \: := \: \frac{1}{2} \int_{\R^2}| \na u |^2 \: - \: \int_{\R^2} u  $$
{\em under the measure constraint:  $| \{ u > 0 \}| = 1\}$}.  Crudely, any minimizer $u$ of $E$ provides a solution $\Omega_u = \{ u > 0 \}$ of the shape optimization  problem, and vice versa. We refer once again to \cite{HenrotPierre} for  rigourous statements. In particular, showing that any minimizer of $E$ is radial shows that any minimizing domain is radial (without {\it a priori} regularity assumption).  

In the case of the fractional laplacian, a close context has been recently investigated by Lopes and Maris in \cite{LP}. Their result is the following:
\begin{proposition}
Let $s\in (0,1)$ and assume that $F,G: \R\to \R$ are such that $u\mapsto F(u)$, $u\mapsto  G(u)$ map $\dot H^s(\R^d)$ into $L^1(\R^d)$. Assume that $u\in \dot H^s(\R^d)$ is a solution of the minimization problem
$$\begin{aligned}
\text{Minimize }E(u):=\int_{\R^d} |\xi|^{2s} |\hat u(\xi)|^2\:d\xi + \int_{\R^d} F(u) \quad 
\text{under the constraint }\int_{\R^d} G(u)=\lambda.   
  \end{aligned}
$$
Then $u$ is radially symmetric.\label{prop:LM}
\end{proposition}
Looking closely at their proof, it seems that their arguments can be extended as such to the functionals
$$ F(u) :=  u,\quad  G(u) :=  \mathbf 1_{u>0} $$
although such $F$ and $G$ do not map $\dot H^{1/2}(\R^2)$ into $L^1(\R^2)$. This is likely to yield the radial symmetry of the minimizing domain for our shape functional $J$ (without {\it a priori} regularity assumption). See  Remark \ref{rmk:rings} for further discussion. 
\medskip

The plan of our paper is the following: Section \ref{sec:preliminaries} collects more or less standard results on the fractional laplacian, which will be used throughout the article. Special attention is paid to regularity properties of solutions of \eqref{EL}, that we deduce from regularity results for the Laplace equations in domains with cracks. Section \ref{sec:derivee-forme} is devoted to the proof of Theorem \ref{prop:der-forme-qcq} and Corollary \ref{thm:derivee-forme}. Finally,  Section \ref{sec:hyp-mobiles} contains the proof of Theorem \ref{thm:hyp-mobiles}.

\section{Preliminaries} \label{sec:preliminaries}

\subsection{Reminders on the fractional laplacian}
We remind here some basic knowledge about $\flap$, see for instance  \cite{Silvestre}. We start with 
\begin{definition}
For any $f \in {\cal S}(\R^n)$, one defines $\flap f$ through the identity
$$ \widehat{\flap f}(\xi) = |\xi| \hat{f}(\xi).  $$ 
\end{definition}
Note that $g=\flap f$ does not belong to  ${\cal S}(\R^n)$ because of the singularity of $|\xi|$ at $0$. Nevertheless, it is $C^\infty$ and satisfies for all $k \in \N$:
$$ \sup_{x\in \R^n} (1 + |x|^{n+1}) |g^{(k)}(x) | < +\infty $$
This allows for a definition of $\flap$ over a large subspace of ${\cal S}'(\R^n)$, by duality. We shall only retain 
\begin{proposition}
$\flap$ extends into a continuous operator from $H^{1/2}(\R^n)$ to $H^{-1/2}(\R^n)$.
\end{proposition}
This is clear from the definition. 

\medskip
It is also well-known that $\flap$ can be identified with the Dirichlet-to-Neumann operator, in the following sense (writing $(x,z) \in \R^n \times \R$ the elements of $\R^{n+1}$ ):
\begin{theorem}
Let $u \in H^{1/2}(\R^n)$. One has 
$ \flap u = - \pa_z U\vert_{z=0}$, where $U$ is the unique solution in 
$$\dot{H}^1(\R^{n+1}_+) \: := \: \left\{ U \in L^2_{loc}(\R^{n+1}_+) , \quad \na U  \in L^2(\R^{n+1}_+) \right\}$$
of 
\begin{equation*}
-\Delta_{x,z} U = 0 \: \mbox{ in }\R^{n+1}_+, \quad U\vert_{z=0} = u.
\end{equation*}
\end{theorem}
Let us remind that the normal derivative $\pa_z U\vert_{z=0} \in H^{-1/2}(\R^n)$ has to be understood in a weak sense: for all $\phi \in H^1(\R^{n+1}_+)$, 
$$ < \pa_z   U\vert_{z=0}, \gamma \phi > \: = \: - \int_{\R^{n+1}_+} \na U \cdot \na \phi $$
where $\gamma$ is the trace operator (onto $H^{1/2}(\R^n)$). It coincides with the standard derivative whenever $U$ is smooth. 

\medskip
We end this reminder with a formula for the fractional laplacian: 
\begin{theorem}
Let $f$ satisfying $\int_{\R^n} \frac{|f(x)|}{1 + |x|^{n+1}} dx < +\infty$, with regularity $C^{1,\eps}$, $\eps > 0$, over an open set $\Omega$. Then, $\flap f$ is continuous over $\Omega$, and for all $x \in \Omega$, one has 
 \begin{eqnarray}
(-\Delta)^{1/2} f(x)&=&C_1 \PV\int_{\R^n}\frac{f(x)-f(y)}{|x-y|^{n+1}}\:dy\\
&=&C_1\int_{\R^n} \frac{f(x)-f(y)- \na f(x)\cdot (x-y)\mathbf 1_{|x-y|\leq C}}{|x-y|^{n+1}}\:dy.
\label{int-lapla-frac}
\end{eqnarray}
for some $C_1 = C_1(n)$ and any $C > 0$. 
\end{theorem}

\subsection{The Dirichlet problem for $\flap$. Regularity properties (2d case).}
\label{ssec:regul-flap}
In view of system \eqref{EL}, a key point in our analysis is to know the behavior near the boundary of solutions to the following fractional Dirichlet problem:
\begin{equation} \label{ELbis}
\begin{aligned}
(-\Delta)^{1/2} u=f \text{ on } \Om,\\
u=0\text{ on }\Om^c,
\end{aligned}
\end{equation}
where $\Omega$ is a smooth open set of $\R^2$ and $f \in  C^\infty(\R^2)$. Note that in system \eqref{ELbis}, we prescribe the value of $u$ not only at $\pa \Omega$, but in the whole $\Omega^c$. This is reminiscent of the non-local character of $\flap$: remind that $u = U\vert_{z=0}$, where $U$ satisfies the (3d) local problem
\begin{equation} \label{3dEL}
\begin{aligned}
-\Delta U = 0 \: \mbox{ in } \: z > 0, \\
\pa_z U = -f  \: \mbox{ over }  \: \Omega \times \{ 0 \}, \quad U = 0   \: \mbox{ over }  \: \Omega^c \times \{0 \},
\end{aligned}
\end{equation}
whose mixed Robin/Dirichlet boundary condition must be specified over the whole plane $\{ z = 0\}$.

\medskip
We were not able to find direct references for regularity properties of problem \eqref{ELbis}, although the $C^{0,1/2}$ regularity of $u$ and the existence of $\pa_n^{1/2} u$ are evoked in \cite{Bogdan, CRS, DSR}. In particular, we could not collect information on transverse and tangential regularity of the solution $u$ near $\pa \Omega \times \{ 0\}$. We shall use results for the Laplace equation in domains with cracks,  in the following way. Let $\eta \in C^\infty_c(\R^3)$, odd in $z$, with $\theta(x,z) = z$ for $x$ in a neighborhood of $\bar \Omega$ and $|z| \le 1$. Then, $V(x,z) := U(x,z) + \eta(x,z) f(x)$ satisfies
\begin{equation} \label{3dELbis}
\begin{aligned}
-\Delta V = F \: \mbox{ in } \: z > 0, \\
\pa_z V = 0  \: \mbox{ over }  \: \Omega \times \{ 0 \}, \quad U = 0   \: \mbox{ over }  \: \Omega^c \times \{ 0 \},
\end{aligned}
\end{equation}
where $F := -\Delta_{x,z}(\eta  f)$ is smooth, odd in $z$, and compactly supported in $\R^3$. We then extend $V$ to $\{z < 0 \}$ by the formula $V(x,z) := -V(x,-z)$, $z < 0$. In this way, we obtain the system 
\begin{equation} \label{3dELter}
\begin{aligned}
-\Delta V = F \: \mbox{ in } \: \R^3 \setminus\left(\Omega \times \{0\}\right) \\
\pa_z V = 0  \: \mbox{ at }  \: \left(\Omega \times \{0\}\right)^{\pm}
\end{aligned}
\end{equation}
which corresponds to a Laplace equation outside a  ``crack'' $\Omega \times \{ 0 \}$ with Neumann boundary condition on each side of the crack. We can now use regularity results for the laplacian in singular domains, such as those of \cite{CoDauDu}. First, note that $V$ is $C^\infty$ away from $\Omega \times \{0\}$, by standard elliptic regularity. Let now  $\Gamma$  be a connected component of $\pa \Omega$, and $\varphi$ a truncation function such that $\varphi = 1$ in a neighborhood of $\Gamma$, with $\Supp \varphi \cap (\pa\Om\setminus\Gamma) = \emptyset$. Then, $\varphi V$ still satisfies a system of type \eqref{3dELter}, with $\Omega$ replaced by $\Supp \varphi$, and $F$ by $F - [\Delta, \varphi ] V$ (which is still smooth). We can then apply  \cite[Theorem A.4.3]{CoDauDu}, which leads to the following
\begin{theorem} \label{the:Dauge}
Let $\Gamma$ be a connected component of $\pa\Om$. 
We denote by $(r,\theta)$ polar coordinates in the planes normal to $\Gamma$ and centered on $\Gamma$, and by $s$ the arc-length on $\Gamma$,   so that
$$
\R^3\setminus( (\Supp \varphi\cap\Om)\times\{0\})=\{(s,r,\theta),\ s\in (0,L),\ r>0,\ \theta\in(-\pi,\pi)\}.
$$
Then, the solution $V$ of \eqref{3dELter} has the following asymptotic expansion, as $r \rightarrow 0$: for any integer $K \ge 0$, 
\begin{equation} \label{asymptoticexp}
V = \sum_{k=0}^K r^{1/2+k}  \Psi^k(s,\theta) \: + \: U_{reg,K} \: + \: U_{rem, K}  
\end{equation}
where: 
\begin{itemize}
\item the coefficients $ \Psi^k$ are regular over $[0,L] \times [-\pi,\pi]$,     
\item $U_{reg,K}$ is regular over $\R^3$,
\item the remainder $U_{rem,K}$ satisfies $\pa^\beta  U_{rem,K} = o(r^{K-|\beta|+1/2})$ for all $\beta \in \N^3$.   
\end{itemize}\label{thm:dev-u}
\end{theorem}
 Setting 
 $$\mbox{ $\psi^k(s) := \Psi^k(s,\pi)$, $u_{rem,K}(x) := U_{rem,K}(x,0^+)$, $u_{reg,K}(x) := U_{reg,K}(x,0^+)$,}$$
 we can get back to the solution $u$ of \eqref{ELbis} and obtain the asymptotic expansion
 \begin{equation} \label{asymptoticexp2}
u = \sum_{k=0}^K r^{1/2+k} \psi^k(s) \: + \: u_{reg,K} \: + \: u_{rem, K}.
\end{equation}
 Such formulas will be at the core of the next sections.

  \section{Shape derivative of the energy $J(\Om)$}
  
  \label{sec:derivee-forme}
  
  This section is devoted to the proof of Theorem \ref{prop:der-forme-qcq} and Corollary  \ref{thm:derivee-forme}. The proof of Theorem \ref{prop:der-forme-qcq} is rather technical, although it follows a simple intuition. Therefore let us first explain how Corollary \ref{thm:derivee-forme} is derived.
  
  Assume that Theorem \ref{prop:der-forme-qcq} holds. Consider the application
  $$
  \mathcal J:\zeta\in \mathcal C^1_b(\R^2, \R^2)\mapsto J((I+ \zeta) \Om),
  $$
  where 
  $$
  \mathcal C^1_b(\R^2, \R^2)= \mathcal C^1\cap W^{1,\infty}(\R^2, \R^2).
  $$
  We recall that $\mathcal C^1_b$ equipped with the norm $\|\cdot \|_{W^{1,\infty}}$ is a Banach space.

  We first claim that $\mathcal J$ is differentiable at $\zeta=0$. This follows from the differentiability of the application 
  $$
  \zeta\in \mathcal C^1_b(\R^2, \R^2)\mapsto  v_\zeta \in H^{1/2}(\R^2),
  $$
  where $v_\zeta= u_\zeta\circ (I+\zeta)$ and $u_\zeta= u_{(I+\zeta)\Om}$ is the solution of the Euler-Lagrange equation associated with  $(I+\zeta) \Om$. The proof goes along the same lines as the one of Lemma \ref{lem:regu0} below. Furthermore, the variational formulation of the Euler-Lagrange equation implies  (see formula \eqref{eq:J})
  $$
  J((I+\zeta)\Om)=-\frac 12 \int_{\R^2} v_\zeta \det (I+\na\zeta).
  $$
  The differentiability of $\mathcal J$ follows.
  
  Using Theorem \ref{prop:der-forme-qcq}, we then identify the differential of $\mathcal J$ at $\zeta=0$. Indeed, if $(\phi_t)_{t\in \R}$ is the flow associated with $\zeta\in \mathcal C^\infty_0(\R^2)$, 
  $$
  \left(\frac{d}{dt}J(\phi_t(\Om))\right)_{|t=0}
  $$
  is the G\^ateaux derivative of $\mathcal J$ at point 0 in the direction $\zeta$. We infer that
  $$
  d\mathcal J(0) \zeta= C_0 \int_{\pa \Om}\zeta\cdot n (\pa_n^{1/2} u_\Om)^2\:d\sigma
  $$
  for all $\zeta\in \mathcal C^\infty_0(\R^2,\R^2)$, and by density for all $\zeta \in \mathcal C^1_b(\R^2, \R^2)$.

  Now, assume that $\Om$ is a bounded domain with $\mathcal C^\infty$ boundary, which minimizes $J$ under the constraint $|\Om|=1$. In other words, $\zeta=0$ is a minimizer of $\mathcal J(\zeta)$ in the Banach space $\mathcal C^1_b$ under the constraint $V(\zeta):=|(I+\zeta)(\Om)|=1$. According to the theorem of Lagrange multipliers, there exists $\lambda\in \R$ such that
  \be\label{lagrange-multi}
  \forall \zeta \in \mathcal C^1_b(\R^2,\R^2),\quad \left(d\mathcal J(0)+\lambda dV(0)\right)\zeta=0.
  \ee
  It is proved in \cite{HenrotPierre} that
  $$
  dV(0)\zeta= \int_{\pa \Om} \zeta\cdot n \;d\sigma.
  $$
  Thus \eqref{lagrange-multi} becomes
  $$
  \exists \lambda\in \R,\ \forall \zeta \in \mathcal C^1_b(\R^2,\R^2),\quad \int_{\pa \Om}( \zeta\cdot n) \left((\pa_n^{1/2} u_{\Om})^2 +\frac{\lambda}{C_0}\right)d\sigma=0.
  $$
  Since $\zeta$ is arbitrary, we infer that $\pa_n^{1/2} u_\Om$ is constant on $\Om$. Moreover, since $u_\Om\geq 0$ on $ \Om$ by the maximum principle, the constant is positive. This completes the proof of Corollary \ref{thm:derivee-forme}.

  \smallskip

  We now turn to the proof of Theorem \ref{prop:der-forme-qcq}.
  In the case of the classical laplacian, the shape derivative of the Dirichlet energy is well-known and is proved in the book by Henrot et Pierre \cite{HenrotPierre}. Let us recall the main steps of the derivation, which will be useful in the case of the fractional laplacian. Let
  $$
  I_f(\Om):=\inf_{u\in H^1_0(\Om)}\left(\frac{1}{2}\int_\Om |\na u|^2 - \int_\Om f u\right),
  $$
  For all $\zeta\in \mathcal C^1\cap W^{1,\infty}(\R^2)$, consider the flow $\phi_t$ associated with $\zeta$. Then for all $t\in \R$, $\phi_t$ is a diffeomorphism of $\R^2$. We recall the following properties, which hold for all $t\in \R$:
  \be\label{phit}
  \begin{aligned}
  \dot \phi_0=\zeta,\quad \left|\det(\na \phi_t(x))\right|=:j(t,x)=\exp\left(\int_0^t(\dv \zeta)(\phi_s(x))\:ds\right)  \\
  \frac{d\phi_t^{-1}}{dt} _{|t=0}=-\zeta,\quad \left|\det(\na \phi_t^{-1})\right|=j(-t,x).
  \end{aligned}
  \ee
  The last line merely expresses the fact that $\phi_t^{-1}= \phi_{-t}$, for all $t\in \R$.
  
  For $t\geq 0$, let $\Om_t=\phi_t(\Om)$, and let $w_t$ be the solution of the associated Euler-Lagrange equation, namely
  $$\begin{aligned}
  -\Delta w_t=f\quad\text{in }\Om_t,\\
  w_t=0\text{ on }\pa \Om_t.   
    \end{aligned}
  $$
  Eventually, we define $z_t:= w_t\circ \phi_t$. 
  
  It is proved in \cite{HenrotPierre} that
  $$
  \left(\frac{d}{dt}I_f(\Om_t)\right)_{|t=0}=-\frac{1}{2}\int_{\pa\Om} (\zeta \cdot n )(\pa_n w_0)^2 d\sigma.
  $$
  Indeed, since $w_t$ solves the Euler-Lagrange equation, we have
  $$
  I_f(\Om_t)=-\frac{1}{2}\int_{\Om_t}f w_t= -\frac{1}{2}\int_{\Om}z_t(y) f\circ\phi_t(y) j(t, y)\:dy.
  $$
  Differentiating $z_t=w_t\circ \phi_t$ with respect to $t$, we obtain
  $$
  \dot z_t=\dot w_t + \dot \phi_t \cdot \na w_t,
  $$
  and thus in particular
  $$
  \dot w_0 + \zeta \cdot n\; \pa_n w_0=0\text{ on }\pa\Om.
  $$
  The Euler-Lagrange equation yields
  $$
  -\Delta \dot w_0=0\quad\text{in }\Om.
  $$
  Gathering all the terms and using the fact that $w_{0|\pa \Om}=0$, we deduce that
  \begin{eqnarray*}
  \left(\frac{d}{dt}I_f(\Om_t)\right)_{|t=0}&=&-\frac{1}{2}\int_\Om( f(\dot w_0 + \zeta \cdot \na w_0 )+ (\zeta \cdot \na f) w_0+f \dv \zeta  w_0)\\
  &=&-\frac{1}{2}\int_\Om (-\dot w_0 \Delta w_0 +\dv (\zeta fw_0) )\\
  &=&\frac{1}{2}\left(\int_\Om w_0 \Delta \dot w_0+\int_{\pa\Om} \dot w_0\pa_n w_0 d\sigma\right)\\
  &=&-\frac{1}{2}\int_{\pa\Om} (\zeta \cdot n) (\pa_n w_0)^2 d\sigma.
  \end{eqnarray*}
  Therefore the shape derivative of $I_f$ is similar to the one of $J_f$, the fractional derivative being merely replaced by a classical derivative.

  \vskip2mm
  
  Unfortunately, the proof in the case of the classical laplacian can only partially be transposed to the fractional laplacian. Indeed, several integration by parts play a crucial role in the computation, and cannot be used in the framework of the fractional laplacian.
  
   We use therefore a different method to estimate $d J_f(\Om_t)/dt$. The main steps of the proof are as follows:
  \begin{enumerate}
   \item As above, we introduce $u_t=u_{\Om_t,f}$ and $v_t=u_t\circ \phi_t$. We derive regularity properties and asymptotic expansions for $v_t$, from which we deduce a decomposition of $u_0$ and $\dot u_0$.
  
  \item In order to avoid the singularities of $\dot u_0$ near $\pa\Om$, we introduce a truncation function $\chi_k$ supported  in $\Om$, and vanishing in the vicinity of the boundary. Using the integral form of the fractional laplacian, we then derive an integral formula for an approximation of $d J_f(\Om_t)/dt$ involving $u_0$, $\dot u_0$ and $\chi_k$.

  \item Keeping only the leading order terms in the decomposition of $u_0$ and $\dot u_0$, we obtain an expression of $d J_f(\Om_t)/dt$ in terms of $\pa_n^{1/2} u_0$, and we prove that this expression is independent of the choice of the truncation function $\chi_k$.
  
  \item We then evaluate the contributions of the remainder terms to the integral formula, and we prove that they all vanish as $k\to \infty$.

  \end{enumerate}
  
 Most of the technicalities are contained in steps 3 and 4. However, some more or less formal calculations - performed at the end of step 2 - lead relatively easily to the desired result. 
  Before tackling the core of the proof, let us introduce some notation: 
  \begin{itemize}
   \item We denote by $\Gamma_1,\cdots, \Gamma_N$ the connected components of $\pa\Om$, and we parametrize each $\Gamma_i$ by its arc-length $s$. We denote by $L_i$ the length of $ \Gamma_i$.

  \item The number $r\in \R$ stands for the (signed) distance to the boundary of $\Om$. More precisely, $|r|$ is the distance to the boundary, and $r>0$ inside $\Om$, $r<0$ outside $\Om$;
  \item We denote by $U_t$ the three dimensional extension of $u_t$ in the half-space, i.e. the function such that
  $$
  -\Delta U_t=0\text{ on }\R^3_+ ,\quad
  U_t\vert_{z=0}= u_t.
  $$
  The three-dimensional function $V_t$ is then defined by
  $$
  V_t(x,z)= U_t(\phi_t(x),z)\quad x\in \R^2, z\in \R,
  $$
  so that $V_t\vert_{z=0}=v_t$.
  
  \item Derivatives with respect to $t$ are denoted with a dot.
  \end{itemize}

  {\setlength{\leftmargini}{0pt}
  
  \subsection{Regularity of $u_0$ and $\dot u_0$ and expansions}

  We start with the following lemma
  \begin{lemma}
   For all $t$ in a neighbourhood of zero,
  \be\label{reg-u0}
  \begin{aligned}
   v_t\in H^{1/2}(\R^2), \quad \dot v_t\in H^{1/2}(\R^2),\\
   u_t\in H^{1/2}(\R^2), \quad \dot u_t\in H^{-1/2}(\R^2).
  \end{aligned}
  \ee
  
  \label{lem:regu0}
  \end{lemma}
  The proof is rather close to the one of Theorem 5.3.2 in \cite{HenrotPierre}. In order to keep the reading as fluent as possible, the details are postponed to the end of the section. The idea is to prove that $V_t$ solves a three-dimensional elliptic equation with smooth coefficients. The implicit function theorem then implies that $t\mapsto V_t\in H^1$ is $\mathcal C^1$ in a neighbourhood of $t=0$. Since $v_t$ is the trace of $V_t$, $t\mapsto v_t\in H^{1/2}$ is also a $\mathcal C^1$ application. Eventually, the chain-rule formula entails that $\dot u_t\in H^{-1/2}$.

  We also derive asymptotic formulas for $V_t$ and $v_t$ in terms of $r$: we rely on the results in the paper by Costabel, Dauge and Duduchava \cite{CoDauDu}, and we use the notations of paragraph \ref{ssec:regul-flap} (see also Theorem \ref{thm:dev-u} of the present paper).
  We claim that there exists $\psi_0, \psi_1, \psi_2\in \mathcal C^1(\pa\Om)$, $u_1, u_2\in W^{1,\infty}(\R^2)$ such that
  \begin{eqnarray}
  \label{dec-u0}u_0&=&\sqrt{r_+}\psi_0(s) + u_1(x),\\ 
  \label{dec-dotu0}\dot u_0&=&\mathbf 1_{r>0}\frac{1}{\sqrt{r}}\psi_1(s) + \sqrt{r_+}\psi_2(s) + u_2(x),
  \end{eqnarray}
  where
  \begin{eqnarray}
  \psi_0(s)=\pa_n^{1/2} u_0(s),\label{def-psi0}\\
  \psi_1(s)=\frac{1}{2}\zeta \cdot n(s) \psi_0(s). \label{def-psi1}
  \end{eqnarray}
  Moreover, there exists $\delta>0$ such that
  \be\label{reg:duO-L1}
  \dot u_t\in L^\infty((-\delta, \delta), L^1(\R^2)).
  \ee
  These decompositions and the regularity result \eqref{reg:duO-L1} will be proved at the end of the section, after the proof of Lemma \ref{lem:regu0}.

  \subsection{An integral formula for an approximation of $d J_f(\Om_t)/dt$}
  
  We first use the Euler-Lagrange equation \eqref{EL} in order to transform the expression defining $J_f(\Om_t)$. Classically, we prove that the unique minimizer $u$ of the energy
  $$
  \frac{1}{2}\left\langle (-\Delta)^{1/2}v,v\right\rangle_{H^{-1/2},H^{1/2}} - \int_{\R^2}fv
  $$
  in the class $\{v\in H^{1/2}(\R^2),\  v_{|\Omega_t^c}=0\}$ satisfies
  $$
  \left\langle (-\Delta)^{1/2}u,v\right\rangle_{H^{-1/2},H^{1/2}} -\int_{\R^2} fv=0 \quad \forall v\in H^{1/2}(\R^2)\text{ s.t. }v_{|\Omega_t^c}=0.
  $$
  Choosing $v\in \mathcal C^\infty_0(\Om_t)$, we infer that $(-\Delta)^{1/2} u=f$ in $\Om$. Since $u_{|\Om_t^c}=0$, we infer that $u=u_{\Om_t, f}=u_t$, and in particular
  $$\langle (-\Delta)^{1/2} u_t, u_t\rangle=\int_{\Om_t} fu_t.$$
  The identity above yields
  $$
  J_f(\Om_t)=- \frac{1}{2}\int_{\R^2}f u_t= - \frac{1}{2}\int_{\Om_t}f u_t.
  $$
  Changing variables in the integral, we obtain
  \be\label{eq:J}
  J_f(\Om_t)= - \frac{1}{2}\int_\Om v_t(y)f\circ \phi_t(y) j(t,y)\:dy.
  \ee
  Since $t\mapsto v_t \in H^{1/2}$ is differentiable, $t\mapsto J_f(\Om_t)$ is $\mathcal C^1$ for $t$ close to zero, and
  $$
  \left(\frac{d}{dt} J_f(\Om_t)\right)_{|t=0}= - \frac{1}{2}\int_{\Om}( \dot v_0f+ \zeta \cdot \na f v_0+v_0f  \dv \zeta  ).
  $$
  Now, for $k\in \N$ large enough, we define $\chi_k\in \mathcal C^\infty_0(\R^2)$ by $\chi_k(x)=\chi(kr)$, where
  $\chi\in\mathcal C^\infty(\R)$ and $\chi(\rho)=0$ for $\rho \leq 1$, $\chi(\rho)=1$ for $\rho \geq 2$.
  Then, since $u_0=v_0$,
  \begin{eqnarray*}
  \left( \frac{d}{dt} J_f(\Om_t)\right)_{|t=0}&=&- \frac{1}{2}\lim_{k\to \infty} \int_{\R^2} \chi_k(\dot v_0f+ v_0\;\dv( \zeta f))\\
  &=&-\frac{1}{2}\lim_{k\to \infty}\left(\int_{\R^2} \chi_k u_0\;\dv (\zeta f)+\left[\frac{d}{dt}\int_{\R^2} \chi_k f v_t\right]_{|t=0}\right)\\
  &=&-\frac{1}{2}\lim_{k\to \infty}\left(\int_{\R^2} \chi_k u_0\;\dv (\zeta f)+\left[\frac{d}{dt}\int_{\R^2} u_t (f\chi_k)\circ \phi_t^{-1}j(-t,\cdot)\right]_{|t=0}\right).
  \end{eqnarray*}
  For fixed $k$ and for $t$ in a neighbourhood of zero, there exists a compact set $K_k$ such that $K_k\Subset \Om$ and $\Supp \chi_k\circ \phi_t^{-1}\subset K_k$.
  Since $\dot u_t\in L^\infty_t(L^1_x)$ according to \eqref{reg:duO-L1}, we can use the chain rule and write
  $$
  \left[\frac{d}{dt}\int_{\R^2} u_t (f\chi_k)\circ \phi_t^{-1}j(-t,\cdot)\right]_{|t=0}=\int_{\R^2}(\dot u_0 f\chi_k -u_0 f\zeta \cdot \na \chi_k- u_0 \chi_k \zeta \cdot \na f-u_0\chi_k f\dv \zeta   ).
  $$
  Using the decomposition \eqref{dec-u0} together with the definition of $\chi_k$, we deduce that
  $$
  \left|  \int_{\R^2} u_0 f \zeta \cdot \na \chi_k\right|\leq C \int_{\R} \sqrt{r_+}k|\chi'(kr)|\:dr\leq \frac{C}{\sqrt{k}}.
  $$
  Notice also that
  $$
  \int_{\R^2}(- u_0 \chi_k \zeta \cdot \na f-u_0\chi_k f\dv \zeta   )=-\int_{\R^2}  u_0 \chi_k\dv(f\zeta ).
  $$
  
  We now focus on the term involving $\du$; since $(-\Delta)^{1/2}u_0=f$ on the support of $\chi_k$, we have
  \begin{eqnarray*}
   \int_{\R^2}\dot u_0 f\chi_k &=&\int_{\R^2} \dot u_0 \chi_k (-\Delta)^{1/2}u_0\\
  &=&\int_{\R^2}u_0 (-\Delta)^{1/2}( \dot u_0 \chi_k ) .
  \end{eqnarray*}
  Notice also that $\chi_k  (-\Delta)^{1/2}( \dot u_0)=0$. Indeed,  $\dot u_t$ is smooth on $K_k$ for $t$ small enough (see for instance \eqref{reg:dut} below).  Hence for $x\in K_k$, the integral formula \eqref{int-lapla-frac} makes sense and  we have, using \eqref{reg:duO-L1},
  \begin{eqnarray*}
  (-\Delta)^{1/2} \dot u_0(x)&=&C_1\int_{\R^2}\frac{\du (x) -\du(y)-\na \du(x)\cdot (x-y)\mathbf{1}_{|x-y|\leq C}}{|x-y|^3}\:dy\\
  &=&C_1\left[\frac{d}{dt}\int_{\R^2}\frac{u_t (x) -u_t(y)-\na u_t(x)\cdot (x-y)\mathbf{1}_{|x-y|\leq C}}{|x-y|^3}\:dy\right]_{|t=0}\\
  &=&\frac{d}{dt}(f(x))=0.
  \end{eqnarray*}
  Eventually, we obtain
  $$
   \int_{\R^2}\dot u_0 f\chi_k=\int_{\R^2}u_0  \left[(-\Delta)^{1/2}, \chi_k \right]\dot u_0.
  $$
  Gathering all the terms, we infer eventually
  $$
  \left(\frac{d}{dt} J_f(\Om_t)\right)_{|t=0}= -\frac{1}{2}\lim_{k\to \infty}\int_{\R^2}u_0  \left[(-\Delta)^{1/2}, \chi_k \right]\dot u_0.
  $$
  Let us now express the right-hand side in terms of the kernel of the fractional laplacian.
  Using the integral formula \eqref{int-lapla-frac} together with the expansion \eqref{dec-dotu0}, we infer that
  $$
    \left[(-\Delta)^{1/2}, \chi_k \right]\dot u_0(x)=C_1\int_{\R^2}\frac{\du(y)(\chi_k(x)-\chi_k(y)) - \du(x)\na \chi_k(x)\cdot (x-y)\mathbf 1_{|x-y|\leq C}}{|x-y|^3}dy.
  $$
  The value of the integral above is independent of the constant $C$.
  Therefore the shape derivative of the energy $J_f$ is given by
  \be
  \label{cond-opti}
  \begin{aligned}
  \frac{d}{dt} J_f(\Om_t)_{|t=0}= -\frac{C_1}{2}\lim_{k\to \infty}I_k,\text{ where}\\
  I_k:= \int_{\R^2\times \R^2}\frac{u_0(x)}{|x-y|^3}\left\{ \du(y)(\chi_k(x)-\chi_k(y)) - \du(x)\na \chi_k(x)\cdot (x-y))\mathbf 1_{|x-y|\leq C}\right\}dx\:dy.
  \end{aligned}
  \ee
  
  The next step is to compute the asymptotic value of $I_k$ as $k\to\infty$. This part is rather technical, and involves several error estimates. However, the intuition leading the calculations is simple: first, the main order is obtained when $u_0$ and $\du$ are replaced by the leading terms in their respective developments. Moreover, because of the truncation $\chi_k$, the integral is concentrated on the boundary $\pa\Om $. All these claims will be fully justified in the next paragraph.
  
  If we follow these guidelines, we end up with
  $$
  I_k\approx\mathrm{PV}\int_{\Td\times\Td}\frac{u_0(x)\du(y)}{|x-y|^3}(\chi_k(x)-\chi_k(y))\:dx\:dy,
  $$
  where $\Td$ is a tubular neighbourhood of $\pa\Om$ of width $\delta\ll 1$ (see \eqref{def:Td}). If we change cartesian coordinates for local ones, and if we neglect the curvature of $\pa\Om$ - which is legitimate if $\delta$ is small - we are led to
  $$
  I_k\approx\mathrm{PV}\int_{[0,\delta]^2}\int_{(0,L)^2}\frac{\psi_0(s)\psi_1(s)}{((s-s')^2+(r-r')^2)^{3/2}}\sqrt{\frac{r}{r'}}(\chi(kr)-\chi(kr')\:ds\:ds'\:dr\:dr'.
  $$
  For simplicity, we have assumed that $\pa \Om$ only has one connected component, of length $L$. Integrating first with respect to $s'$, and changing variables by setting $\rho=kr$, $\rho'=kr'$, we obtain eventually
  \begin{eqnarray*}
  I_k&\approx&2 C \left(\int_0^L \psi_0\psi_1\right)\mathrm{PV}\int_0^\infty\int_0^\infty \sqrt{\frac{\rho}{\rho'}}\frac{\chi(\rho)- \chi(\rho')}{|\rho-\rho'|^2}\:d\rho\:d\rho'\\
  &\approx& C \left(\int_0^L \psi_0\psi_1\right) \int_0^\infty\int_0^\infty \frac{\chi(\rho)- \chi(\rho')}{\sqrt{\rho \rho'}(\rho-\rho')}\:d\rho\:d\rho',
  \end{eqnarray*}
  where 
  $$
  C=\int_0^\infty \frac{dz}{(1+z^2)^{3/2}}.
  $$
  There remains to prove that the value of the integral in the right-hand side does not depend on $\chi$, which we do at the end of the next paragraph, and the formula of Theorem \ref{prop:der-forme-qcq} is proved.
  
  Of course, the above calculation is very sketchy, and careful justification must be given at every step. But the general direction follows these formal arguments.
  
  \subsection{Asymptotic value of $I_k$}

  We now evaluate the integral $I_k$ defined in \eqref{cond-opti}. There are two main ideas:
  \begin{itemize}
   \item We prove that the domain of integration can be restricted to a tubular neighbourhood of $\pa\Om$ (see lemma \ref{lem:tubular}).
  
  \item Since the integral $I_k$ is bilinear in $u_0,\du$, we replace $u_0$ and $\du$ by their expansions in \eqref{dec-u0}, \eqref{dec-dotu0}. The leading term is obtained when $u_0$ and $\du$ are replaced by the first terms in their respective developments. We prove in the next subsection that all other terms vanish as $k\to \infty$.
  
  \end{itemize}

  We begin with the following Lemma (of which we postpone the proof):
  
  \begin{lemma}
  For $\delta>0$, let 
  \be\label{def:Td}
  \Td:=\{ x\in \R^2,\ d(x,\pa\Om)<\delta\}.
  \ee
  Choose $C=\frac{\delta}{2}$ in the definition of $I_k$ \eqref{cond-opti}. Then there exists a constant $C_\delta$ such that for $k>5/\delta$,
  $$
  \left|  \int_{(\Td\times \Td)^c}\frac{u_0(x)}{|x-y|^3}\left\{ \du(y)(\chi_k(x)-\chi_k(y)) - \du(x)\na \chi_k(x)\cdot (x-y))\mathbf 1_{|x-y|\leq \delta/2}\right\}dx\:dy\right|\leq \frac{C_\delta}{\sqrt{k}}.$$

  \label{lem:tubular}
  \end{lemma}
  We henceforth focus our attention on the value of the integral on $\Td\times \Td$.
  Replacing $u_0$ and $\du$ by the first terms in the expansions \eqref{dec-u0}, \eqref{dec-dotu0}, we define
  \begin{multline}\label{def:Ind}
  I_k^\delta:=\int_{\Td \times \Td} \frac{\psi_0(s)\sqrt{r_+}}{|x-y|^3}\left[\psi_1(s')\frac{\mathbf{1_{r'>0}}}{\sqrt{r'}}(\chi(kr)-\chi(kr') )\right. \\
  \left. -\psi_1(s)\frac{\mathbf{1_{r>0}}}{\sqrt{r}} k\chi'(kr) n(s)\cdot (x-y)\mathbf  1_{|x-y|\leq \delta/2}\right]dx\:dy.
  \end{multline}
  There is a slight abuse of notation in the integral above, since we use simultaneously cartesian and local coordinates. In order to be fully rigorous, $r,r',s,s'$ should be replaced by $r(x),\ r(y),\ s(x), s(y)$ respectively. However, the first step of the proof will be to express the integral $I_k^\delta$ in local coordinates, and therefore we will avoid these heavy notations.

  In fact, the computation of the limit of $I_k^\delta$ is much more technical than the estimation of all other quadratic remainder terms, which we will achieve in the next subsection. We now prove the following:

  \begin{lemma}
  There exists an explicit constant $C_2$, independent of $\chi$ and of $\Om$, such that
  $$
   \lim_{\delta\to 0}  \lim_{k\to \infty}I_k^\delta=C_2\sum_{i=1}^N\int_0^{L_i} \psi_0(s)\psi_1(s)\:ds.
  $$

  \label{lem:limite-opti}
  \end{lemma}

  \begin{proof}
  Throughout the proof, different types of error terms will appear, which we have gathered in Lemma \ref{lem:error} below. We will therefore refer to \eqref{error1}, \eqref{error2}, \eqref{error3}, \eqref{error4} to classify the different types of error terms.
  
  Several preliminary simplifications are necessary:
  \begin{itemize}
  \item We use local coordinates instead of cartesian ones, i.e. we change $x$ into $(s,r)$ and $y$ into $(s',r')$. Since $|r|,|r'|\leq \delta$, the jacobian of this change of coordinates is, for $\delta>0$ small enough,
  $$
  \left|1+ r\kappa(s)\right|\: \left|1+r'\kappa(s')\right|= \left(1+r\kappa(s)\right)\left(1+r'\kappa(s')\right),
  $$
  where $\kappa$ is the algebraic curvature of $\pa \Om$. We refer to the Appendix for a simple proof. Notice that this jacobian is always bounded.

  \item We write $\Td= \cup_{i=1}^N \Td^i$, where
  $$
  \Td^i= \{ x\in \R^2,\ d(x,\Gamma_i)<\delta\}.
  $$
  
  If $\delta$ is small enough, $\Td^i\cap \Td^j=\emptyset$ for $i\neq j$, and $|x-y|$ is bounded from below for $x\in \Td^i, \: y\in\Td^j$ by a constant independent of $\delta$. Hence, for $i\neq j$,
  \begin{eqnarray*}
  &&\left|\int_{\Td^i\times\Td^j} \frac{\psi_0(s)\sqrt{r_+}}{|x-y|^3}\left[\psi_1(s')\frac{\mathbf{1_{r'>0}}}{\sqrt{r'}}(\chi(kr)-\chi(kr') )\right. \right.\\
  &&\left.\qquad\qquad\qquad\qquad\qquad\left. -\psi_1(s)\frac{\mathbf{1_{r>0}}}{\sqrt{r}} k\chi'(kr) n(s)\cdot (x-y)\mathbf  1_{|x-y|\leq \delta/2}\right]dx\:dy\right|\\
  &\leq & C \int_{\Td^i\times\Td^j}\mathbf{1_{r,r'>0}}\frac{\sqrt{r}}{\sqrt{r'}}\:dx\:dy\\
  &\leq & C\int_{(0, \delta)^2} \frac{\sqrt{r}}{\sqrt{r'}}\:dr\:dr'\\
  &\leq& C \delta^2.
  \end{eqnarray*}
  Therefore, in the integral defining $I_k^\delta$, we replace the domain of integration $\Td\times \Td$ by $\cup_{i=1}^N\Td^i\times\Td^i$, 
  and this introduces an error term of order $\delta^2.$
  
  \item In local coordinates, we write $\Td^i$ as $(0, L_i)\times (-\delta, \delta)$. We replace the jacobian $$\left(1+r\kappa(s)\right)\left(1+r'\kappa(s')\right)$$ by $(1+r\kappa(s))^2$. Since $s,s'\in (0,L_i)$ (i.e. $x$ and $y$ belong to the neighbourhood of the same connected component of $\pa \Om$), this change introduces error terms bounded by \eqref{error1}, \eqref{error3}, which vanish as $k\to \infty$ and $\delta\to 0$. More details will be given in the fourth step for similar error terms; we refer to Lemma \ref{lem:error}.

  \item We evaluate $|x-y|$ in local coordinates for $x,y\in \Td^i$. We have
  $$
  \begin{aligned}
  x=p(s)- rn(s),\\ 
  y=p(s') - r'n(s'),
  \end{aligned}
  $$
  where $p(s)\in \R^2$ is the point of $\pa \Om$ with arc-length $s$.
  Since the boundary $\Gamma_i$ is $\mathcal C^\infty$, 
  $$
  p(s)-p(s')=(s-s')\tau(s) + O(|s-s'|^2),
  $$
  where $\tau(s)$ is the unit tangent vector at $p(s)\in \pa \Om$.
  Using the Frenet-Serret formulas, we also have
  $$
  n(s)-n(s')=-(s-s')\kappa(s)\tau(s) + O(|s-s'|^2),
  $$
  where $\kappa$ is the curvature of $\Gamma_i$. Gathering all the terms, we infer that
  \be\label{x-y}
  x-y=(s-s')(1+\kappa (s)r)\tau(s) + (r-r')n(s)+ O(|s-s'|^2+|r-r'|^2),
  \ee
  so that
  \be\label{x-y-norm}
  |x-y|^2=(s-s')^2(1+\kappa(s)r)^2 + (r-r')^2+ O\left(|s-s'|^3+ |r-r'|^3\right)
  \ee
  and
  $$
  |x-y|^{-3}= \left((s-s')^2(1+\kappa(s)r)^2 + (r-r')^2\right)^{-3/2}\left(1+ O(|s-s'|+|r-r'|)\right).
  $$
  In particular, there exists a constant $C$ such that
  $$
  \frac{1}{|x-y|^3}\leq \frac{C}{ \left((s-s')^2 + (r-r')^2\right)^{3/2}},
  $$
  and replacing $|x-y|^{-3}$ by $ \left((s-s')^2(1+\kappa(s)r)^2 + (r-r')^2\right)^{-3/2}$ generates yet another error term bounded by \eqref{error1}, \eqref{error3}.

  \item We replace the factor $\psi_1(s')$ by $\psi_1(s)$; this also leads to an error term of the type \eqref{error1}.
  
  \item Using \eqref{x-y}, we infer that
  $$
  n(s)\cdot (x-y)=r-r'+ O(|s-s'|^2 + |r-r'|^2).
  $$
  The second term in the right-hand side of the above equality gives rise to an error term of the type \eqref{error3}.

  \item The last preliminary step is to replace the indicator function $\mathbf 1_{|x-y|\leq \delta/2}$ in \eqref{def:Ind} by a quantity depending on $s,s',r,r'$. Using the asymptotic development \eqref{x-y-norm} above, it can be easily proved that there exists a constant $c$ such that
  $$
  \left|\mathbf 1_{|x-y|\leq \delta/2}- \mathbf 1_{(s-s')^2(1+\kappa(s)r)^2 + (r-r')^2\leq \delta^2/4}\right| \leq \mathbf 1_{(1-c\delta)\frac{\delta}{2}\leq |x-y| \leq (1+c\delta)\frac{\delta}{2}}.
  $$
  Therefore the substitution between the two indicator functions yields an error term bounded by
  \begin{eqnarray*}
  &&\sum_{i=1}^N\int_{\Td^i\times \Td^i}\frac{dx\:dy}{|x-y|^2}|\na \chi_k(x)|  \mathbf 1_{(1-c\delta)\frac{\delta}{2}\leq |x-y| \leq (1+c\delta)\frac{\delta}{2}}\\
  &\leq &\sum_{i=1}^N \int_{\Td^i}|\na \chi_k(x)| \:dx\int_{\R^2} \frac{dz}{|z|^2} \mathbf 1_{(1-c\delta)\frac{\delta}{2}\leq |z| \leq (1+c\delta)\frac{\delta}{2}}\\
  &\leq & C \ln \left(\frac{1+c\delta}{1-c\delta}\right)\leq C \delta.
  \end{eqnarray*}

  \end{itemize}
  
  As a consequence, at this stage, we have proved that
  \begin{eqnarray}
  I_k^\delta\label{limit-1}
  &=&\sum_{i=1}^N\int_0^\delta\int_{-\delta}^\delta \int_{(0,L_i)^2} \frac{\psi_0(s)\psi_1(s)\sqrt{r}(1+r\kappa(s))^2}{\left((s-s')^2(1+\kappa(s)r)^2 + (r-r')^2\right)^{3/2}}\times\\&\times&\left[\frac{\mathbf{1_{r'>0}}}{\sqrt{r'}}(\chi(kr)-\chi(kr') ) -\frac{1}{\sqrt{r}} k\chi'(kr) (r-r')\mathbf  1_{(s-s')^2(1+\kappa r)^2 + (r-r')^2\leq \delta^2/4}\right]ds\,ds\,'dr'\,dr\nonumber\\
  &+&O\left( \frac{\ln k}{\sqrt{k}}+ \delta |\ln\delta|\right).\nonumber
  \end{eqnarray}
  
  We now evaluate the right-hand side of the above identity. We first prove that the term involving $\chi'(kr)$ does not contribute to the limit, due to symmetry properties of the integral. This was expected, since this term had a vanishing integral in the beginning; its only role was to ensure the convergence of $I_k$. We then focus on the term involving $\chi(kr)-\chi(kr')$, and we prove that its asymptotic value is independent of $\chi$.

  First, since the integral \eqref{limit-1} is convergent, we have
  $$
  \eqref{limit-1}=\sum_{i=1}^N\lim_{\eps\to 0} \int_0^\delta\int_{-\delta}^\delta \int_{(0,L_i)^2}\mathbf 1_{|r-r'|\geq \eps}\ \cdots
  $$
  Therefore, up to the introduction of a truncation, we can separate the two terms of \eqref{limit-1}. We have in particular
  \begin{eqnarray}
  &&\int_{-\delta}^\delta\frac{\psi_0(s)\psi_1(s)\mathbf{1}_{|r-r'|\geq \eps}(1+r\kappa(s))^2}{\left((s-s')^2(1+\kappa (s)r)^2 + (r-r')^2\right)^{3/2}} k\chi'(kr) (r-r')\mathbf  1_{(s-s')^2(1+\kappa r)^2 + (r-r')^2\leq \delta^2/4}dr'\nonumber\\
  &=& -\int_{-\delta-r}^{\delta-r} \frac{\psi_0(s)\psi_1(s)\mathbf{1}_{|\xi|\geq \eps}(1+r\kappa(s))^2}{\left((s-s')^2(1+\kappa(s)r)^2 + \xi^2\right)^{3/2}} k\chi'(kr) \xi\;\mathbf  1_{(s-s')^2(1+\kappa r)^2 + \xi^2\leq \delta^2/4}d\xi.\label{limit-1-bis}
  \end{eqnarray}
  Notice that the above integral only bears on the values of $\xi$ such that $|\xi|\leq \delta/2$. On the other hand, for all $r$ such that $\chi'(kr)\neq 0$, we have $1/k\leq r\leq 2/k$, so that if $k>4/\delta$,
  $$\begin{aligned}
     \delta-r\geq \delta - \frac{2}{k}>\frac{\delta}{2},\\
  -\delta-r<-\delta<-\frac{\delta}{2}.
   \end{aligned}
  $$
  Hence the integral \eqref{limit-1-bis} is in fact equal to
  $$
  \psi_0(s)\psi_1(s)(1+r\kappa(s))^2k\chi'(kr) \int_{-\delta/2}^{\delta/2}\frac{ \xi\mathbf  1_{|\xi|\geq \eps}\mathbf 1_{(s-s')^2(1+\kappa r)^2 + \xi^2\leq \delta^2/4}}{\left((s-s')^2(1+\kappa (s) r)^2 + \xi^2\right)^{3/2}}\;d\xi.
  $$
  Since the integrand is odd in $\xi$, the integral is identically zero for all $\eps>0$ and for all $\delta, k$ such that $k\delta>4$.

  There remains to investigate the first term in \eqref{limit-1}, namely
  \be\label{limit-2}
  \int_{(0,\delta)^2} \int_{(0,L_i)^2} \frac{\psi_0(s)\psi_1(s)(1+r\kappa(s))^2(\chi(kr)-\chi(kr') )}{\left((s-s')^2(1+\kappa(s)r)^2 + (r-r')^2\right)^{3/2}}\sqrt{\frac{r}{r'}} \mathbf 1_{|r-r'|\geq \eps}ds\:ds'\:dr'\:dr.
  \ee
  We first symmetrize the integral by exchanging the roles of $r$ and $r'$. We have 
  \begin{eqnarray*}
  && \sqrt{\frac{r}{r'}}(1+r\kappa(s))^2 - \sqrt{\frac{r'}{r}}(1+r'\kappa(s))^2\\& =& \frac{r-r'}{\sqrt{rr'}}(1+r\kappa(s))^2 + O\left(\sqrt{\frac{r'}{r}}|r-r'|\right),\\
  &&\frac{1}{\left((s-s')^2(1+\kappa(s)r)^2 + (r-r')^2\right)^{3/2}}- \frac{1}{\left((s-s')^2(1+\kappa(s)r')^2 + (r-r')^2\right)^{3/2}}\\&=& O\left(\frac{|r-r'|}{\left((s-s')^2+ (r-r')^2\right)^{3/2}}\right).\end{eqnarray*}
  The term 
  $$
  \int_{(0,\delta)^2} \int_{(0,L_i)^2} \left|\psi_0(s)\psi_1(s)(\chi(kr)-\chi(kr') )\right| \sqrt{\frac{r'}{r}} \frac{|r-r'|}{\left((s-s')^2+ (r-r')^2\right)^{3/2}}ds\:ds'\:dr'\:dr
  $$
  is bounded by an error term of the type \eqref{error1}, and is therefore $O(\ln k/\sqrt{k})$ for all $\eps>0$.
  As a consequence,
  \begin{eqnarray}
  &&\lim_{\eps\to 0}\eqref{limit-2}\nonumber\\
  &=& \frac{1}{2} \int_{(0,\delta)^2} \int_{(0,L_i)^2} \frac{\psi_0(s)\psi_1(s)(1+r\kappa(s))^2(\chi(kr)-\chi(kr') )}{\left((s-s')^2(1+\kappa(s)r)^2 + (r-r')^2\right)^{3/2}}\frac{r-r'}{\sqrt{rr'}}ds\,ds'\,dr'\,dr\label{limit-3}\\
  &&+ O\left( \frac{\ln k}{\sqrt{k}}\right).\nonumber
  \end{eqnarray}
  We now compute the integral with respect to $s'\in (0, L_i)$. 
  Setting 
  $$
  \phi(\xi):=\int_0^\xi\frac{dz}{(1+z^2)^{3/2}},
  $$
  we have 
  \begin{eqnarray*}
  &&\int_0^{L_i}\frac{ds'}{\left((s-s')^2(1+\kappa(s)r)^2 + (r-r')^2\right)^{3/2}}\\
  &=&\frac{1}{(r-r')^2(1+\kappa(s)r)}\left( \phi \left(\frac{(L_i-s)(1+\kappa(s) r)}{|r-r'|} \right)+ \phi\left(\frac{s(1+\kappa(s) r)}{|r-r'|}\right)\right). 
  \end{eqnarray*}
  Inserting this formula into the integral above and changing variables, we obtain
  \begin{multline*}
   \eqref{limit-3}= \frac{1}{2} \int_{(0,k\delta)^2} \int_0^{L_i}\psi_0(s) \psi_1(s) \frac{\chi(\rho)-\chi(\rho')}{(\rho-\rho')\sqrt{\rho\rho'}}\left(1+\kappa(s)\frac{\rho}{k}\right)\times\\
  \times\left[\phi \left(\frac{k(L_i-s)(1+\kappa(s)\rho/k)}{|\rho-\rho'|} \right)+ \phi\left(\frac{ks(1+\kappa(s)\rho/k)}{|\rho-\rho'|}\right)\right]ds\:d\rho\:d\rho'. 
  \end{multline*}
  Since
  $$
  \begin{aligned}
  \int_0^\infty \int_0^\infty \left| \frac{\chi(\rho)-\chi(\rho')}{(\rho-\rho')\sqrt{\rho\rho'}}\right|d\rho\:d\rho'<\infty,\\
  \text{and }0\leq \phi(\xi)\leq \phi(+\infty)<\infty\quad\forall \xi>0,
  \end{aligned}
  $$
  using Lebesgue's theorem, we infer that for all $\delta>0$,
  $$
  \lim_{k\to \infty}  \eqref{limit-3}=\phi(+\infty)\left(\int_0^{L_i}\psi_0(s) \psi_1(s) ds\right) \int_0^\infty \int_0^\infty \frac{\chi(\rho)-\chi(\rho')}{(\rho-\rho')\sqrt{\rho\rho'}}d\rho\:d\rho'.
  $$
  There only remains to prove that the integral involving $\chi$ is in fact independent of $\chi$. We have
  \begin{eqnarray*}
  I_0&:= & \int_0^\infty \int_0^\infty \frac{\chi(\rho)-\chi(\rho')}{(\rho-\rho')\sqrt{\rho\rho'}}d\rho\:d\rho'\\
  &=& \int_0^\infty \int_0^\infty\int_0^1 \chi'(\tau\rho+(1-\tau)\rho')\frac{d\tau\:d\rho\:d\rho'}{\sqrt{\rho\rho'}}\\
  &\underset{z=\rho'/\rho}{=}&\int_0^\infty \int_0^\infty\int_0^1\chi'(\rho(\tau+(1-\tau)z))\frac{d\tau\:d\rho\:dz}{\sqrt{z}}\\
  &=&\int_0^\infty\int_0^1\frac{1}{\tau+(1-\tau)z}\frac{d\tau\:dz}{\sqrt{z}}\\
  &=&\int_0^\infty\frac{\ln z}{(z-1)\sqrt{z}}dz.
  \end{eqnarray*}
  Gathering all the terms, we infer that
  $$\lim_{k\to \infty}  I_k^\delta=  I_0\phi(+\infty)\sum_{i=1}^N\int_0^{L_i}\psi_0\psi_1+ O\left( \delta |\ln\delta|\right).$$
  Passing to the limit as $\delta\to 0$, we obtain the result announced in Lemma \ref{lem:limite-opti}.

  \end{proof}

  \subsection{Evaluation of remainder terms in the integral $I_k$}

  $\triangleright$ We start with the proof of Lemma \ref{lem:tubular}. The idea is to divide $\R^2\times \R^2\setminus \Td\times \Td$ into subdomains and to evaluate the contribution of every subdomain.
  \begin{itemize}
   \item For $(x, y)\in (\Td^c\cap \Om)^2$, $\chi_k(x)=\chi_k(y)=1$ and $\na \chi_k(x)=0$ for $k>\delta^{-1}$, so that the contribution of $ \Td^c\cap \Om\times  \Td^c\cap \Om$ is zero.
  
  \item For $x\in \Td^c\cap \Om^c$, $u_0(x)=0$: the contribution of  $ \Td^c\cap \Om^c\times  \R^2$ is zero.
  
  \item For $x\in \Td^c\cap \Om , y \in \Td^c \cap \Om^c $, $\du(y)=0$ and $\na\chi_k(x)=0$, so that the contribution of $ \Td^c\cap \Om\times\Td^c \cap \Om^c$ is zero.
  
  \item For $x\in \Td^c\cap \Om$, $y\in \Td$, $\chi_k(x)=1$ and $\na \chi_k(x)=0$, so that the contribution of the sub-domain is 
  \begin{eqnarray*}
  &&\left|  \int_{\Td^c\cap \Om\times \Td}\frac{u_0(x)}{|x-y|^3}\left\{ \du(y)(\chi_k(x)-\chi_k(y)) - \du(x)\na \chi_k(x)\cdot (x-y)\mathbf 1_{|x-y|\leq \delta/2}\right\}dx\:dy\right| \\
  &\leq & C\int_{\Td^c\cap \Om\times \Td}\frac{1}{|x-y|^3}| \du(y)| (1-\chi_k(y))dx\:dy.
  \end{eqnarray*}
  Since $1-\chi_k$ is supported in $\Om^c \cup \{d(y,\pa \Om)\leq 2/k\}$ and $|\du(y)|\leq C \frac{\mathbf 1_\Om}{\sqrt{d(y,\pa \Om)}}$, the right-hand side is bounded by
  $$
  \frac{C}{|\delta-\frac{2}{k}|^3}\int_0^{2/k}\frac{dr}{\sqrt{r}}\leq \frac{C_\delta}{\sqrt{k}}.
  $$

  \item For $x\in \Td$, $y\in \Td^c$, the integral
  $$
  \int_{\Td \times \Td^c}\frac{u_0(x)\du(x)}{|x-y|^3}\na \chi_k(x)\cdot (x-y)\mathbf 1_{|x-y|\leq \delta/2}\:dx\:dy
  $$
  is in fact supported in 
  $$
  (\{d(x,\pa \Om)\leq 2/k\}\times \{d(y,\pa \Om)>\delta\})\cap\{|x-y|\leq \delta/2\}.
  $$
  It is easily seen that for $k$ large enough (say $k>5/\delta$) this set is empty, and therefore  the integral is zero.

  \item There only remains
  $$
  \int_{\Td \times \Td^c}\frac{u_0(x)}{|x-y|^3} \du(y)(\chi_k(x)-\chi_k(y)) \:dx\:dy.
  $$
  Using the same kind of estimates as for the domain $\Td^c\cap \Om\times \Td$, it can be proved that this term is bounded by $C_\delta k^{-3/2}$.\qed
  \end{itemize}

  $\triangleright$ We now estimate the remainder terms on $\Td\times \Td$. We use the following lemma:

  \begin{lemma}
  Let $k\geq 1$ and $\delta>0$ such that $k\delta>5$.
  
  Then the following estimates hold: for $1\leq i \leq N$,
  \begin{align}
   \int_{T_\delta^i\times T_\delta^i} \mathbf 1_{r,r'>0}\frac{\sqrt{r}}{\sqrt{r'}(\sqrt{r}+\sqrt{r'})} \frac{1}{|x-y|^2}|\chi_k(x)-\chi_k(y)|\:dx\:dy=O\left( \frac{\ln k}{\sqrt{k}}\right),\label{error1}\\
  \int_{T_\delta^i\times T_\delta^i} \frac{\sqrt{r_+}}{|x-y|^3}\left|\chi_k(x)-\chi_k(y)- \na \chi_k(x)\cdot(x-y)\mathbf{1}_{|x-y|\leq \delta/2}\right|\:dx\:dy=O\left(\sqrt{\delta }+ \frac{\ln k}{\sqrt{k}}\right)\label{error2},\\
  \int_{T_\delta^i\times T_\delta^i}\frac{|\na \chi_k(x)|}{|x-y|}dx\:dy=O(\delta|\ln \delta|),\label{error3}\\
  \int_{T_\delta^i\times T_\delta^i}\mathbf{1}_{r>0, r'<0}\frac{\sqrt{r}}{|x-y|^3}|\chi_k(x)-\chi_k(y)|\:dxdy=O\left(\frac{1}{ k^{1/2}}\right)\label{error4}.
  \end{align}

   \label{lem:error}

  \end{lemma}
  
  \begin{proof}
  Throughout the proof, we use the following facts:
  \begin{itemize}
  \item The jacobian of the change of variables $(x,y)\to(s,r,s',r')$ is
  $$
  \left(1+ r\kappa(s)\right)\left(1+ r'\kappa(s')\right)\leq C;
  $$
   \item Using the expansion \eqref{x-y-norm}, we infer that there exists a constant $C$ such that
  $$
  \frac{1}{|x-y|}\leq \frac{C}{((s-s')^2+(r-r')^2)^{1/2}}.
  $$
  \item For all $\alpha\geq 1$, there exists a constant $C_\alpha$ such that for all $s\in(0,L_i)$, for all $r,r'\in (-\delta, \delta)$ with $r\neq r'$,
  \be\label{in:int-s'}
  \int_0^{L_i} \frac{ds'}{((s-s')^2+(r-r')^2)^{\alpha/2}}\leq C_\alpha \left\{\begin{array}{ll}
                                                                           |r-r'|^{1-\alpha}&\text{ if }\alpha>1,\\
  									|\ln |r-r'|\: |&\text{ if }\alpha=1.
                                                                          \end{array}\right.
  \ee
  Indeed, for $\alpha>1$
  \begin{eqnarray*}
   \int_0^{L_i} \frac{ds'}{((s-s')^2+(r-r')^2)^{\alpha/2}}&\leq & \int_\R\frac{dz}{(z^2+(r-r')^2)^{\alpha/2}}\\
  &\underset{\xi=z/|r-r'|}{\leq} &|r-r'|^{1-\alpha}\int_{\R}\frac{d\xi}{(1+\xi^2)^{\alpha/2}},\end{eqnarray*}
  while for $\alpha=1$
  \begin{eqnarray*}
   \int_0^{L_i} \frac{ds'}{((s-s')^2+(r-r')^2)^{1/2}}&=&\int_{-\frac{s}{|r-r'|}}^{\frac{L_i-s}{|r-r'|}}\frac{d\xi}{(1+\xi^2)^{1/2}}\\
  &\leq & \int_{-\frac{L_i}{|r-r'|}}^{\frac{L_i}{|r-r'|}}\frac{d\xi}{(1+\xi^2)^{1/2}}\\
  &\leq & 2\ln \left(\frac{L_i}{|r-r'|}\right).
  \end{eqnarray*}
  
  \end{itemize}
  We now tackle the proof of \eqref{error1}-\eqref{error4}.
  
  \begin{enumerate}
  \item \textit{Proof of \eqref{error1}:} we split the domain $(\Td\times \Td)\cap\{r,r'\geq 0\}$ into four subdomains:
  \begin{itemize}
   \item $0\leq r, r'\leq 3/k$;
  \item $0\leq r'\leq 3/k$ and $3/k<r<\delta$;
  \item $0\leq r\leq 3/k$ and $3/k<r'<\delta$;
  \item  $3/k<r, r'<\delta$.
  \end{itemize}
  Since $\chi_k(x)=1$ for $r>2/k$, it is easily seen that $\chi_k(x)-\chi_k(y)=0$ on the last subdomain. Moreover, we obviously have
  \begin{eqnarray*}
  &&\int_{s,s'\in \Gamma_i} \int_{0\leq r\leq 3/k,\;3/k<r'<\delta}\frac{\sqrt{r}}{\sqrt{r'}(\sqrt{r}+\sqrt{r'})} \frac{1}{|x-y|^2}|\chi_k(x)-\chi_k(y)|\:dx\:dy\\
  &\leq & \int_{s,s'\in \Gamma_i} \int_{0\leq r'\leq 3/k,\;3/k<r<\delta}\frac{\sqrt{r}}{\sqrt{r'}(\sqrt{r}+\sqrt{r'})} \frac{1}{|x-y|^2}|\chi_k(x)-\chi_k(y)|\:dx\:dy.
  \end{eqnarray*}
  Therefore we focus on the estimates of the integral on the first two subdomains above. First, since $\chi_k$ is a Lispchitz function with a Lipschitz constant of order $k$, we have
  \begin{eqnarray*}
  &&\int_{s,s'\in \Gamma_i} \int_{0\leq r,r'\leq 3/k}\frac{\sqrt{r}}{\sqrt{r'}(\sqrt{r}+\sqrt{r'})} \frac{1}{|x-y|^2}|\chi_k(x)-\chi_k(y)|\:dx\:dy\\ 
  &\leq& C k\int_{s,s'\in \Gamma_i}\int_{0\leq r,r'\leq 3/k}\frac{\sqrt{r}}{\sqrt{r'}(\sqrt{r}+\sqrt{r'})}\frac{1}{|x-y|}\:dx\:dy\\
  &\leq & Ck\int_{0\leq r,r'\leq 3/k}\int_{(0, L_i)^2}\frac{\sqrt{r}}{\sqrt{r'}(\sqrt{r}+\sqrt{r'})} \frac{1}{((s-s')^2+(r-r')^2)^{1/2}}ds\:ds'\:dr\:dr'\\
  &\leq & Ck\int_{0\leq r,r'\leq 3/k}\frac{\sqrt{r}}{\sqrt{r'}(\sqrt{r}+\sqrt{r'})}|\ln |r-r'|\:|\:dr\:dr'\\
  &\leq & Ck\left(\int_{0}^{3/k}\frac{dr'}{\sqrt{r'}}\right)\left(\int_{-3/k}^{3/k}|\ln |z|\:|\:dz\right)\\
  &\leq& Ck\times \frac{1}{\sqrt{k}}\times \frac{\ln k}{k}.
  \end{eqnarray*}
  As for the second subdomain, since $\chi_k(x)=1$ for $r>2/k$, in fact, the domain of integration in $r'$ is only $r'<2/k$. Since $\chi_k$ is bounded by $1$, we have
  \begin{eqnarray*}
   &&\int_{s,s'\in \Gamma_i}\int_{0\leq r'\leq 3/k,\;3/k<r<\delta}\frac{\sqrt{r}}{\sqrt{r'}(\sqrt{r}+\sqrt{r'})}\frac{1}{|x-y|^2}|\chi_k(x)-\chi_k(y)|\:dx\:dy\\
  &\leq &C \int_{0\leq r'\leq 2/k,\;3/k<r<\delta}\int_{(0,L_i)^2}\frac{\sqrt{r}}{\sqrt{r'}(\sqrt{r}+\sqrt{r'})}\frac{1}{(s-s')^2+(r-r')^2}ds\:ds'\:dr\:dr'\\
  &\leq & C\int_{0\leq r'\leq 2/k,\;3/k<r<\delta}\frac{\sqrt{r}}{\sqrt{r'}(\sqrt{r}+\sqrt{r'})}\frac{1}{|r-r'|}dr\:dr'\\
  &\leq & C\left(\int_0^{2/k}\frac{dr'}{\sqrt{r'}}\right)\left(\int_{1/k}^\delta \frac{dz}{z}\right)\\
  &\leq & C\frac{\ln k}{\sqrt{k}}.
  \end{eqnarray*}

  \item \textit{Proof of \eqref{error2}:} we use the same type of domain decomposition as above; the only difference lies in the fact that $r'$ may take negative values.
  \begin{itemize}
   \item On the subdomain $r\geq 3/k, r'\geq 3/k$, the integral is identically zero;
  
  \item On the subdomain $3/k<r<\delta$, $-\delta<r'<3/k$, we have $\na\chi_k(x)=0$ and as soon as $r'\geq 2/k$,
  $$
  \chi_k(x)-\chi_k(y)-\na \chi_k(x)\cdot(x-y)\mathbf{1}_{|x-y|\leq \delta/2}=0.
  $$
   Therefore
  \begin{eqnarray*}
   &&\int_{\substack{{3/k<r<\delta,-\delta<r'<3/k},\\{s,s'\in \Gamma_i}} }{\frac{\sqrt{r_+}}{|x-y|^3}}\left|\chi_k(x)-\chi_k(y)- \na \chi_k(x)\cdot(x-y)\mathbf{1}_{|x-y|\leq \delta/2}\right|\:dx\:dy\\
  &\leq &C\int_{3/k<r<\delta,-\delta<r'<2/k} \int_{(0,L_i)^2}{\frac{\sqrt{r_+}}{((s-s')^2+(r-r')^2)^{3/2}}}ds\:ds'\:dr\:dr'\\
  &\leq & C\int_{3/k<r<\delta,-\delta<r'<2/k}\frac{\sqrt{r}}{|r-r'|^2}\:dr\:dr'\\
  &\leq & C\int_{3/k}^\delta \frac{\sqrt{r}}{r-\frac{2}{k}}dr.
  \end{eqnarray*}
  Using the inequality 
  $$
  \sqrt{r}\leq \sqrt{r-\frac{2}{k}}+ \sqrt{\frac{2}{k}},
  $$
  we infer eventually that the integral is bounded by
  $
  C\left(\sqrt{\delta }+ \frac{\ln k}{\sqrt{k}}\right).
  $
  
  \item On the subdomain $0\leq r\leq 3/k$, $|r'|\geq 3/k$, we use the bound
  $$
  |\chi_k(x)-\chi_k(y)-\na \chi_k(x)\cdot (x-y)\mathbf{1}_{|x-y|\leq \delta/2}|\leq Ck |x-y|,
  $$
  so that
  \begin{eqnarray*}
  &&\int_{\substack{{0\leq r\leq 3/k, 3/k<|r'|<\delta},\\{s,s'\in \Gamma_i}} } {\frac{\sqrt{r}}{|x-y|^3}}\left|\chi_k(x)-\chi_k(y)- \na \chi_k(x)\cdot(x-y)\mathbf{1}_{|x-y|\leq \delta/2}\right|\:dx\:dy\\ 
  &\leq & Ck\int_{0\leq r\leq 3/k, 3/k<|r'|<\delta} {\frac{\sqrt{r}}{|r-r'|}}dr\:dr'\\
  &\leq & -Ck\int_0^{3/k}\sqrt{r}\ln\left(\frac{3}{k}-r\right)\:dr\\
  &\leq & Ck\times\frac{1}{\sqrt{k}}\times\frac{\ln k}{k}.
  \end{eqnarray*}
  
  \item On the subdomain $0\leq r\leq 3/k, |r'|\leq 3/k$, we use a Taylor-Lagrange expansion for the function $\chi_k$, which yields
  $$
  |\chi_k(x)-\chi_k(y)-\na \chi_k(x)\cdot (x-y)\mathbf{1}_{|x-y|\leq \delta/2}|\leq Ck^2 |x-y|^2,
  $$
  so that
  \begin{eqnarray*}
  &&\int_{\substack{{0\leq r\leq 3/k, |r'|\leq 3/k},\\{s,s'\in \Gamma_i}} } {\frac{\sqrt{r_+}}{|x-y|^3}}\left|\chi_k(x)-\chi_k(y)- \na \chi_k(x)\cdot(x-y)\mathbf{1}_{|x-y|\leq \delta/2}\right|\:dx\:dy\\ 
  &\leq & Ck^2\int_{0\leq r\leq 3/k, |r'|\leq 3/k}\sqrt{r}\left| \ln |r-r'|\right|\:dr\:dr'\\
  &\leq & Ck^2 \frac{\ln k}{k} k^{-3/2}.
  \end{eqnarray*}
  
  \end{itemize}
  
  \item  \textit{Proof of \eqref{error3}:} this term is easier to estimate than \eqref{error1}, \eqref{error2}. We merely have
  \begin{eqnarray*}
  \int_{T_\delta^i\times T_\delta^i}\frac{|\na \chi_k(x)|}{|x-y|}dx\:dy
  &\leq & C \int_{(-\delta,\delta)^2}k|\chi'(kr)|\; \left|\ln|r-r'|\:\right|\:dr\:dr'\\
  &\leq & C \int_{-2\delta}^{2\delta}\left|\ln|z|\:\right|\:dz\\
  &\leq & C \delta |\ln \delta|.
  \end{eqnarray*}

  \item \textit{Proof of \eqref{error4}:} since $\chi_k(y)=0$ if $r'<0$, we have
  \begin{eqnarray*}
   &&\int_{T_\delta^i\times T_\delta^i}\mathbf{1_{r>0, r'<0}}\frac{\sqrt{r}}{|x-y|^3}|\chi_k(x)-\chi_k(y)|\:dxdy\\
  &=&\int_{T_\delta^i\times T_\delta^i}\mathbf{1_{r>0, r'<0}}\frac{\sqrt{r}}{|x-y|^3}\chi_k(x)\:dxdy\\
  &\leq &C\int_{-\delta}^0\int_0^\delta \frac{\sqrt{r}}{|r-r'|^2}\chi(kr)\:dr\:dr'\\
  &\leq & C \int_0^{2/k} \sqrt{r}\left(\frac{1}{r}- \frac{1}{r+\delta}\right)\:dr\\
  &\leq & \frac{C}{\sqrt{k}}.
  \end{eqnarray*}

  \end{enumerate}

  \end{proof}

  We now address the rest of the proof of Theorem \ref{thm:derivee-forme}.
  With the help of Lemma \ref{lem:error}, the estimation of the remainder terms in $I_k$ is immediate. We set
  $$
  \begin{aligned}
  w_0(x):=\mathbf 1_{r>0} \frac{\psi_1(s)}{\sqrt{r}},\\
  w_1(x):= \sqrt{r_+}\psi_2(s) + u_2(x),
  \end{aligned}
  $$
  so that $\dot u_0=w_0+w_1$. We write
  \begin{eqnarray*}
  && \dot u_0(y) (\chi_k(x)-\chi_k(y))-\du(x)\na \chi_k(x)\cdot(x-y)\mathbf{1}_{|x-y|\leq \delta/2}\\
  &=&\du(x)\left[\chi_k(x)-\chi_k(y)- \na \chi_k(x)\cdot(x-y)\mathbf{1}_{|x-y|\leq \delta/2}\right]\\
  &&+(\du(y)-\du(x))(\chi_k(x)-\chi_k(y)).
  \end{eqnarray*}
  There exists a constant $C$ such that
  $$
  |u_1(x)\du(x)| + |\sqrt{r_+}\psi_0(s)w_1(x)|\leq C \sqrt{r_+}\quad\forall x\in \R^2,
  $$
  so that
  \begin{eqnarray*}
  &&\left|\int_{\Td\times \Td}\frac{u_1(x)\du(x)}{|x-y|^3} \left[\chi_k(x)-\chi_k(y)- \na \chi_k(x)\cdot(x-y)\mathbf{1}_{|x-y|\leq \delta/2}\right]\:dx\:dy \right| \\
  &+&\left|\int_{\Td\times \Td}\frac{\sqrt{r_+}\psi_0(s)w_1(x)}{|x-y|^3} \left[\chi_k(x)-\chi_k(y)- \na \chi_k(x)\cdot(x-y)\mathbf{1}_{|x-y|\leq \delta/2}\right]\:dx\:dy \right|\\
  &\leq & \eqref{error2}.
  \end{eqnarray*}
  On the other hand, simple calculations show that
  \begin{eqnarray*}
  |\du(x)-\du(y)|&\leq& C|x-y|\left(1+\frac{\mathbf{1}_{r,r'>0}}{\sqrt{rr'}(\sqrt{r}+ \sqrt{r'})}\right)\\
  &&+C\mathbf{1}_{rr'<0}\frac{1}{\sqrt{r_+}+ \sqrt{r'_+}},\\
  |w_1(x)-w_1(y)|&\leq& C|x-y|\left(1+\frac{\mathbf{1}_{r,r'>0}}{\sqrt{r}+ \sqrt{r'}}\right)\\
  &&+C\mathbf{1}_{rr'<0}(\sqrt{r_+}+ \sqrt{r'_+}).
  \end{eqnarray*}
  
  As a consequence,
  $$
  \begin{aligned}
  |u_1(x)|  |\du(x)-\du(y)|\leq C|x-y|\left(r_+ + \mathbf{1}_{r,r'>0}\frac{\sqrt{r}}{\sqrt{r'}(\sqrt{r }+ \sqrt{r'})}\right)+C \mathbf{1_{r>0, r'<0}}\sqrt{r},\\
  \left|\sqrt{r_+}\psi_0(s) \right| \:|w_1(x)-w_1(y)|\leq C|x-y|\left(\sqrt{r_+} + \mathbf{1}_{r,r'>0}\frac{\sqrt{r}}{\sqrt{r}+ \sqrt{r'}}\right)+C \mathbf{1_{r>0, r'<0}}r,
  \end{aligned}
  $$
  so that
  \begin{eqnarray*}
   &&\left|\int_{\Td\times \Td}\frac{u_1(x)}{|x-y|^3}(\du(x)-\du(y))(\chi_k(x)-\chi_k(y))\:dx\:dy\right|\\
  &+&\left|\int_{\Td\times \Td}\frac{\sqrt{r_+}\psi_0(s)}{|x-y|^3}(|w_1(x)-w_1(y))(\chi_k(x)-\chi_k(y))\:dx\:dy\right|\\
  &\leq & \eqref{error1} + \eqref{error4}.
  \end{eqnarray*}

  Gathering all the terms, we deduce that
  $$
  \lim_{\delta\to 0}\lim_{k\to \infty} I_k=C_2\sum_{i=1}^N\int_0^{L_i} \psi_0\psi_1,
  $$
  where $C_2$ is an explicit positive constant. 
  Moreover, using formulas \eqref{def-psi0}, \eqref{def-psi1}, we have
  $$
  \psi_0(s)\psi_1(s)= -\frac{1}{2}(\pa_n^{1/2} u_0(s))^2 \zeta \cdot n(s),
  $$
  where $\zeta(x)=\dot \phi_0(x)$. Eventually, we deduce that
  $$
  \frac{d J_f(\Om_t)}{dt}_{|t=0}= - \frac{C_1C_2}{4 }\sum_{i=1}^N\int_0^{L_i}(\pa_n^{1/2} u_0(s))^2 \zeta \cdot n(s)ds= - \frac{C_1C_2}{4 }\int_{\pa \Om} (\pa_n^{1/2} u_0(s))^2 \zeta \cdot n(s)d\sigma(s).
  $$
  Therefore Theorem \ref{prop:der-forme-qcq} is proved, with $C_0:=-(C_1C_2)/4$. Notice that $C_0$ does not depend on $\Om$.

  \subsection{Proofs of Lemma \ref{lem:regu0} and formulas \eqref{dec-u0}, \eqref{dec-dotu0}}

  \begin{itemize}
   \item The proof of \eqref{reg-u0} follows closely the one of Theorem 5.3.2 in \cite{HenrotPierre}. The only differences come from the mixed boundary conditions and the fact that $\phi_t$ only affects horizontal variables.

  The key point is to prove that $V_t$ is differentiable with respect to $t$ with values in $H^1(\R^3_+)$. To that end, observe that $V_t$ is the solution of the elliptic problem
  \be\label{eq:Vt}
  \begin{aligned}
  -\dv(A_t \na V_t) = 0\text{ in  }\R^3_+,\\
  \pa_z V_t=f\circ \phi_t\text{ on }\Om\times \{0\},\\
  V_t=0\text{ on } \Om^c\times\{0\},
  \end{aligned}
  \ee
  where
  $$
  A_t= \det (\na \phi_t)\begin{pmatrix}
       (\na \phi_t)^{-1}(\na \phi_t^T)^{-1}&0\\
  	0&1
       \end{pmatrix}.
  $$
  (Notice that $\na \phi_t$ is a $2\times 2$ matrix.)
  
  Indeed, \eqref{eq:Vt} is easily proved by writing the variational formulation associated with the equation on $U_t$ and performing changes of variables. Since the latter are  strictly identical to the ones of \cite{HenrotPierre}, we skip the proof.
  
  For further reference, we also write the system derived by lifting the Neumann boundary condition. We set $V_t(x,z)= \tilde V_t (x,z)+ f(\phi_t(x) )\eta(\phi_t(x), z), $ where $\eta\in \mathcal C^\infty_0(\R^3)$ is such that $\eta(x,z)=z$ for $x$ in a neighbourhood of $\Om$ and $|z|\leq 1$. Then $\tilde V_t$ solves
  \be\label{eq:tV}
  \begin{aligned}
  -\dv(A_t \na \tilde V_t)=  |\det (\na \phi_t)| (\Delta (f\eta))(\phi_t(x),z) \quad\text{on }\R^3_+,\\
  \pa_z \tilde V_t=0\text{ on }\Om\times\{0\},\quad \tilde V_t=0\text{ on }\Om^c\times\{0\}.
   \end{aligned}
   \ee

  Let
  $$
  \mathcal V:= \left\{V  \in L^2_\text{loc}(\R^3_+), \ \na V\in L^2(\R^3_+),\
  V=0\text{ on } \Om^c\times\{0\}\right\}.
  $$
  Then according to the Hardy inequality in $\R^3$, there exists a constant $C_H$ such that  for all $V\in \mathcal V$, $V(1 + |x|^2 + |z|^2)^{-1/2}\in L^2(\R^3_+)$ and
  $$
  \int_{\R^3_+}\frac{|V(x,z)|^2}{1 + |x|^2 + |z|^2}\:dx\:dz\leq C_H \int_{\R^3_+} |\na V|^2.
  $$
  This inequality is usually stated in the whole space, but a simple symmetry argument shows that it remains true in the half-space. Therefore $\|\na V\|_{L^2}$ is a norm on the Hilbert space $\mathcal V$.
  
  Now, for $t$ in a neighbourhood of zero and 
  $\tilde V\in \mathcal V$, define the linear form $F(t,\tilde V)\in \mathcal V'$ by
  $$
  \forall W\in \mathcal V, \langle W, F(t,\tilde V)\rangle= \int_{\R^3_+} (A_t \na \tilde V)\cdot \na W - \int_{\R^3_+} |\det (\na \phi_t)|(\Delta (f\eta))(\phi_t(x,z)) W(x,z)\:dx\:dz.
  $$ 
  Notice that $F(t,\tilde V)=0$ is the variational formulation associated with the equation \eqref{eq:tV}. We then claim that the operator
  $$
  F:(t,\tilde V)\in \R \times \mathcal  V \mapsto F(t,\tilde V)\in \mathcal V'
  $$
  has $\mathcal C^1$ regularity for $t$ small enough. Indeed, $
  t\mapsto A_t\in L^\infty(\R^2, \mathcal M_3)
  $
  is $\mathcal C^\infty$ for $t$ in a neighbourhood of zero. On the other hand, for $(A, \tilde V)\in L^\infty(\R^2, \mathcal M_3)\times \mathcal V $, $W$, let
  $$
  a(A,\tilde V): W\in \mathcal V\mapsto \int_{\R^3_+} (A \na \tilde V)\cdot \na W - \int_{\R^3_+} |\det (\na \phi_t)|(\Delta (f\eta))(\phi_t(x,z)) W(x,z)\:dx\:dz.
  $$
  Then the application
  $$
  (A, V)\in L^\infty(\R^2, \mathcal M_3)\times \mathcal V \mapsto a(A,V)\in \mathcal V',
  $$
  has $\mathcal C^1$ regularity since it is the sum of a  bilinear and continuous function and a constant term (with respect to $A,V$).

  Let $\tilde V_0 \in \mathcal V$ be the solution of \eqref{eq:tV} for $t=0$, i.e.
  $$\begin{aligned}
  -\Delta \tilde V_0= \Delta (f\eta),\\
  \pa_z \tilde V_0=0\text{ on }\Om\times\{0\},\quad \tilde V_0=0\text{ on }\Om^c\times\{0\}.
  \end{aligned}
  $$
  Now, for $ W \in\mathcal V$,
  $
  d_V F(0,\tilde V_0)W
  $
  is the linear form
  $$
  W'\in \mathcal V \mapsto \int_{\R^3_+} \na W'\cdot\na W,
  $$
  i.e. the scalar product on $\mathcal V$. Therefore the differential
  $$
  d_V F(0,\tilde V_0):\mathcal V\to \mathcal V'
  $$
  is an isomorphism. The implicit function theorem implies that there exists a $\mathcal C^1$ function $t\mapsto \tilde V(t)$ in a neighbourhood of zero such that
  $$
  F(t,\tilde V(t))=0.
  $$
  Uniqueness for equation \eqref{eq:tV} yields $\tilde V_t=\tilde V(t)$. We infer immediately that $t\mapsto V_t$ is $\mathcal C^1$ in a neighbourhood of zero. We now use the following trace result: for any $U\in \mathcal V$,
  \be\label{trace}
  \|U\vert_{z=0}\|_{H^{1/2}(\R^2)}\leq C(\Om) \|\na U \|_{L^2(\R^3_+)}.
  \ee
  Indeed, it is well known that if $U\in \mathcal V $, the following estimates hold (with constants which do not depend on $\Om$)
  $$\begin{aligned}
  \|U\vert_{z=0}\|_{\dot H^{1/2}(\R^2)}\leq C \| \na U \|_{L^2(\R^3_+)},\\
  \|   U\vert_{z=0}\|_{L^4(\R^2)}\leq C \|\na U\|_{L^2(\R^3_+)}.
    \end{aligned}
  $$
  There only remains to prove that the $L^2$ norm of $U\vert_{z=0}$ is bounded by $\|\na U\|_{L^2}$. Since $U\vert_{z=0}=0$ on $\Om^c$, we have
  \begin{eqnarray*}
  \| U\vert_{z=0}\|_{L^2(\R^2)}&\leq & |\Om|^{1/2} \|   U\vert_{z=0}\|_{L^4(\R^2)}\\
  &\leq & C(\Om)\|\na U\|_{L^2(\R^3_+)}.
  \end{eqnarray*}
  Therefore \eqref{trace} is proved. As a consequence, $t\mapsto v_t\in H^{1/2}(\R^2)$ is  $\mathcal C^1$ in a neighbourhood of zero. The result on $\dot u_t$ follows almost immediately by differentiating the formula
  $$
  u_t=v_t\circ \phi_t^{-1}
  $$
  with respect to $t$. The only difficulty comes from the fact that $v_t$ is \textit{a priori }not smooth enough with respect to $x$ in order to use the chain rule. However, we can write
  $$
  u_t-u_0=(v_t-v_0)\circ \phi_t^{-1} + \left(v_0\circ \phi_t^{-1}-v_0\right).
  $$
  Since $\dot v_t\in \mathcal C([-\delta, \delta],H^{1/2}(\R^2))$ for $\delta>0$ small enough, the first term is $O(t)$ in $L^2(\R^2)$. As for the second one, if $v_0$ were smooth, say $v_0\in H^1(\R^2)$, then we could write
  $$
  \frac{d}{dt}v_0\circ \phi_t^{-1}= \left(\frac{d \phi_t^{-1}}{dt}\right)\cdot \na v_0\circ \phi_t^{-1}.
  $$
  It can be easily checked that the right-hand side is bounded in $H^{-1/2}$ by $C \|v_0\|_{H^{1/2}}$, where $C$ is a constant depending only on $\phi_t$. Therefore if $v_0\in H^1$, $t>0$,
  $$\|v_0\circ \phi_t^{-1}-v_0\|_{H^{-1/2}(\R^2)}\leq C t \|v_0\|_{H^{1/2}(\R^2)}.
  $$
  By density, this inequality remains true for all $v_0\in H^{1/2}$. We conclude that $\dot u_t$ belongs to $\mathcal C([-\delta, \delta], H^{-1/2}(\R^2))$.

  \item We end this section by proving the asymptotic expansions \eqref{dec-u0}, \eqref{dec-dotu0}. We apply the results of Chapter C in \cite{CoDauDu} to the function $V_t$. We start by extending $V_t$ on $\R^3\setminus \Om\times\{0\}$ by setting
  $$
  V_t(x,z)=-V_t(x,-z)\quad x\in \R^2,\ z<0.
  $$
  Without any loss of generality, we assume that the function $\eta\in \mathcal C^\infty_0(\R^3)$ defined in the proof of Lemma \ref{lem:regu0} is odd with respect to $z$.
  
  Then the extended function $V_t$ satisfies
  $$
  \begin{aligned}
  V_t=\tilde V_t +f(\phi_t(x)) \eta(\phi_t(x),z),\\
  -\dv (A_t\na \tilde V_t)=|\det (\na \phi_t)|\Delta(f \eta)(\phi_t(x),z)\quad\text{on }\R^3\setminus \Om\times\{0\}\\
  \pa_z \tilde V_t=0\quad \text{on } \Om\times\{0\}.
  \end{aligned}
  $$
By elliptic regularity, $\tilde V_t\in \mathcal C^\infty (K)$ for any compact set $K\subset  \R^3\setminus \Om\times\{0\}$ such that $K\cap \pa \Om=\emptyset$.
  For every connected component $\Gamma_i$ of $\pa\Om$, we introduce a truncation function $\varphi_i\in \mathcal C^\infty_0(\R^2)$ such that 
  $$\begin{aligned}
  \varphi_i\equiv 1\text{ on a neighbourhood of }\Gamma_i,\\
  \Supp \varphi_i\cap (\pa\Om\setminus \Gamma_i)=\emptyset.
   \end{aligned}
  $$
  Since the support of $\na \varphi_i$ is separated from the zones where $\tilde V_t$ has singularities, we infer that for all $t$, $\varphi_i(x) \tilde V_t (x,z)$ solves a boundary value problem 
  which is elliptic of order two in the sense of Agmon, Douglis, Nirenberg (see \cite{ADN1}, \cite{ADN2} and the definitions in chapter 7 of \cite{Dauge}), with homogeneous boundary conditions of order one and with $\mathcal C^\infty$ right-hand side. Moreover, we have proved in Lemma \ref{lem:regu0} that $V_t\in H^1_{loc}(\R^3\setminus \Om\times\{0\})$. Therefore we can apply Corollary C.6.5 in \cite{CoDauDu}. As in paragraph \ref{ssec:regul-flap}, we denote by $(r,\theta)$ polar coordinates in the planes normal to $\Gamma_i$ and centered on $\Gamma_i$, and by $s$ the arc-length on $\Gamma_i$,   so that
  $$
  \R^3\setminus( \Supp \varphi_i \cap \Omega \times\{0\})=W_\pi=\{(s,r,\theta),\ s\in (0,L_i),\ r>0,\ \theta\in(-\pi,\pi) \}.
  $$
  We deduce that for all $i$, there exist functions $\tilde \psi_i^0(t; s, \theta)$, $\tilde\psi_i^1(t;s,\theta)$, which are smooth with respect to $s,\theta\in [-\pi,\pi]$, such that
  \be\label{dec:Vt}
  \varphi_i(x)\tilde V_t(x,z)=\tilde \psi_i^0(t; s, \theta) r^{1/2} + \tilde\psi_i^1(t; s, \theta) r^{3/2}+ u_{reg}(t;x,z)+ u_{rem}(t;x,z),
  \ee
  with $u_{reg}(t;\cdot) \in \mathcal C^\infty(\R^3)$ for all $t$ and $u_{rem}(t;\cdot)\in \mathcal C^2(\overline{ W_\pi})$ with $\pa^\beta u_{rem}(t;x,z)=o(r^{3/2-|\beta|})$ for all $t$ and for any multi-index $\beta$. Setting
  $$
  V^1_i(t;x,z):= \tilde\psi_i^1(t; s, \theta) r^{3/2}+ u_{reg}(t;x,z)+ u_{rem}(t;x,z) + \eta(\phi_t(x),z),
  $$
  we have $V^1_i\in \mathcal C^1(\overline{ W_\pi})$,  $V^1_i(t;x,z)=O(r)$ for all $t$, and 
  $$
  V_t(x,z)=\tilde \psi^0(t; s, \theta) r^{1/2} + V^1_i(t;x,z)
  $$
  for $x$ in a neighbourhood of $\Gamma_i$.
   Since 
  $$
  u_0(x)=\lim_{z\to 0, z>0} V_0(x,z),
  $$
  we obtain decomposition \eqref{dec-u0} with 
  $$
  \psi^0(s):=\tilde \psi^0_i(0,s,\pi)\text{ for }s\in \Gamma_i,\\
  u_1(x)=\lim_{z\to 0^+}V^1_i(0;x,z)\text{ for }x\in \Supp \varphi_i.
  $$

  Differentiating  \eqref{dec-dotu0} requires regularity results with respect to $t$ on the terms of the decomposition \eqref{dec:Vt}. Using for instance Theorem C.6.2 in \cite{CoDauDu}, or looking precisely at the details of the proof in section C of \cite{CoDauDu}, it is easily proved that the functions $\tilde \psi^0$, $\tilde \psi^1$, $u_{reg}$  and $u_{rem}$ are differentiable with respect to $t$, and that their $t$-derivatives are smooth with respect to $x$. Denoting by $s(x)$, $r(x)$ the local coordinates of a point $x\in \R^2$, we have
  $$
  u_t(x)=\tilde \psi^0(t,s(\phi_t^{-1}(x)), \pi) (r(\phi_t^{-1}(x)))^{1/2} + V^1(t, \phi_t^{-1}(x), 0^+)
  $$
  so that
  \begin{eqnarray}
  \nonumber\dot u_t&=&\left(\pa_t \tilde \psi^0 + \pa_t(\phi^{-1}_t (x))\cdot \na s(\phi_t^{-1}(x))\pa_s \tilde\psi^0\right)(t,s(\phi_t^{-1}(x)), \pi) (r(\phi_t^{-1}(x)))^{1/2} \\
  \label{reg:dut}&&+\frac{1}{2} \pa_t(\phi^{-1}_t (x))\cdot\na r (\phi_t^{-1}(x))\frac{1}{r(\phi_t^{-1}(x))^{1/2}} \\
  \nonumber&&+ \pa_t V^1(t, \phi_t^{-1}(x), 0^+) +  \pa_t(\phi^{-1}_t (x))\cdot\na V^1(t, \phi_t^{-1}(x), 0^+).
  \end{eqnarray}
  Therefore $\dot u_t \in L^\infty_t(L^1_x)$, which proves \eqref{reg:duO-L1}.
  We recall that $\phi_0=\phi_0^{-1}=\mathrm{Id}$, and that $\zeta= \dot \phi_0$. Notice also that since $n$ is the outward pointing normal, $\na r =-n$. We infer 
  $$
  \du=\frac{\zeta \cdot n}{2 \sqrt{r}} + \psi_2(s)\sqrt{r}+ u_2(x),
  $$
  where 
  $$
  \begin{aligned}
  \psi_2(s) =\left(\pa_t \tilde \psi^0 +\zeta\cdot\tau \pa_s \tilde\psi^0\right)(0,s, \pi),\\
  u_2(x)=\pa_t V^1(0, x, 0^+) + \zeta\cdot\na V^1(0, x, 0^+).
  \end{aligned}
  $$
  Thus \eqref{dec-dotu0} is proved.
  
  \end{itemize}

\section{Radial symmetry by the moving plane method}  \label{sec:hyp-mobiles} 
The aim of this section is to prove Theorem \ref{thm:hyp-mobiles}.
This theorem extends the well-known theorem of Serrin \cite{Serrin} on the classical laplacian, for which \eqref{overdetermined} is replaced by
\begin{equation} \label{overdetermined2}
\left\{
\begin{aligned}
-\Delta u = 1, \quad & x \in \Omega, \\
 u = 0, \quad & x \in \pa \Omega, \\
 \pa_n u = c_0, \quad & x \in \pa \Omega 
\end{aligned}
\right.
\end{equation}
The proof of Serrin uses the celebrated moving plane method, and we will adapt it to our fractional setting. 

\subsection{Reminders on the moving plane method} \label{reminders}
We remind here the main arguments of Serrin's proof. The starting point is the introduction of a family of hyperplanes (lines in our 2d case), say $ H_\la := \{ x \in \R^2, x_1 = \la\}$, parametrized by $\la \in \R$. We also define, for any function $u$ defined on $\R^2$, $R_\lambda u(x_1, x_2) \: := \: u(2\la -x_1, x_2)$ the  reflection of $u$ with respect to $H_\la$. 
For $\la$ small enough, $H_\la$ does not intersect the domain $\Omega$. Increasing $\la$, that is moving $H_\la$ from left to right, one reaches a first contact position, corresponding to 
$$ \la_0 := \inf \{\la, H_\la\cap \Omega \neq \emptyset\}. $$
Up to a translation, we can always assume that $\la_0 = 0$. Thus, for $\la >  0$, we can consider the  cap 
$ \Si_\la := \Omega \cap \{x_1 < \la\}$
and its reflection $\Si'_\la$ with respect to $H_\la$. Note that, for $\la > 0$ small enough, $\Si'_\la$ is non-empty and included in $\Omega$. As $\la$ increases, it remains included in $\Omega$ at least until one of the following two geometric configurations is reached (see Figure \ref{fig:configurations}):   
\begin{enumerate}
\item $\Si'_{\la}$ is internally tangent to the boundary of $\Omega$ at some point $P$ not on $H_{\la}$. 
\item $H_{\la}$ is orthogonal to the boundary of $\Omega$ at some point $Q$. 
\end{enumerate}
\begin{figure}
\begin{tabular}{cc}
 \includegraphics[height=9.5cm]{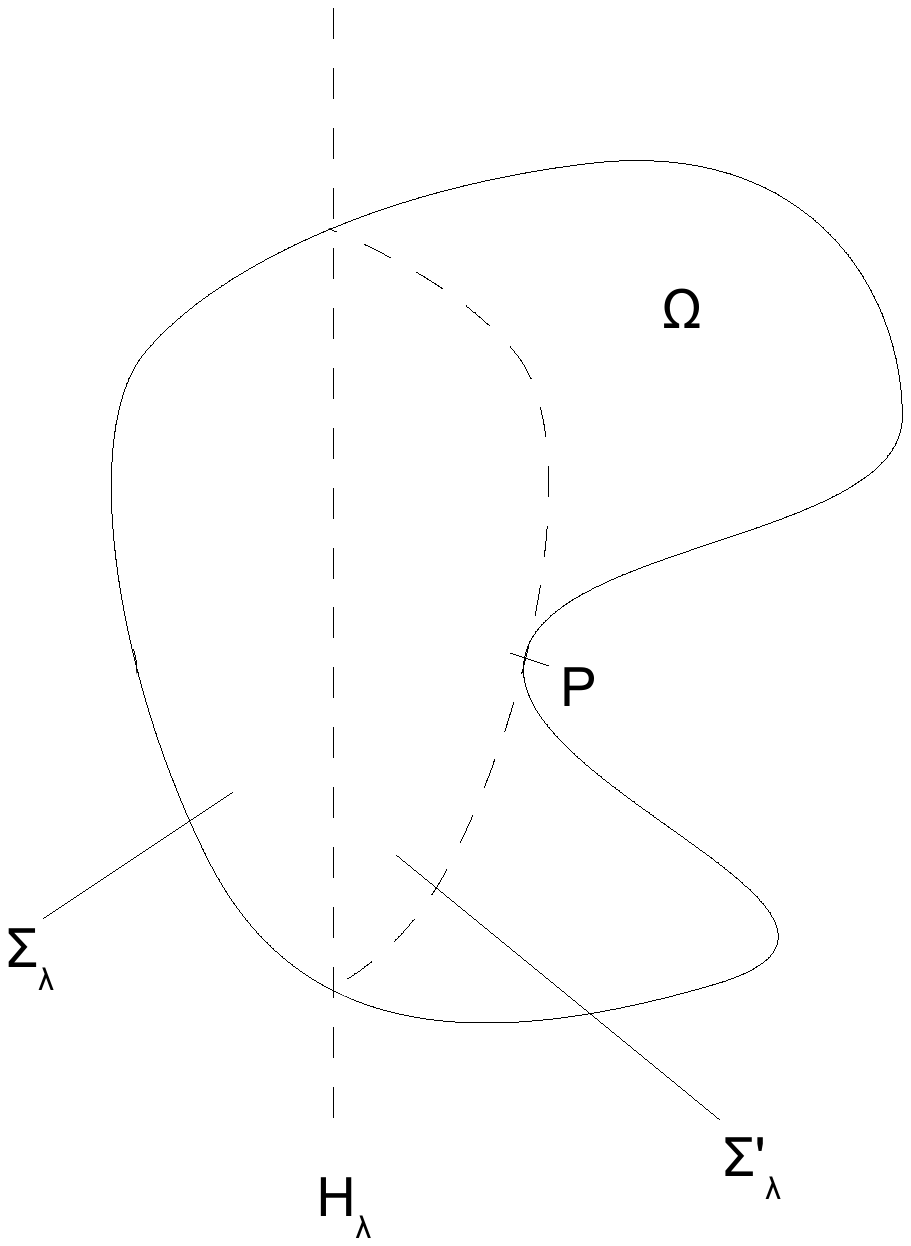}& \includegraphics[height=9.5cm]{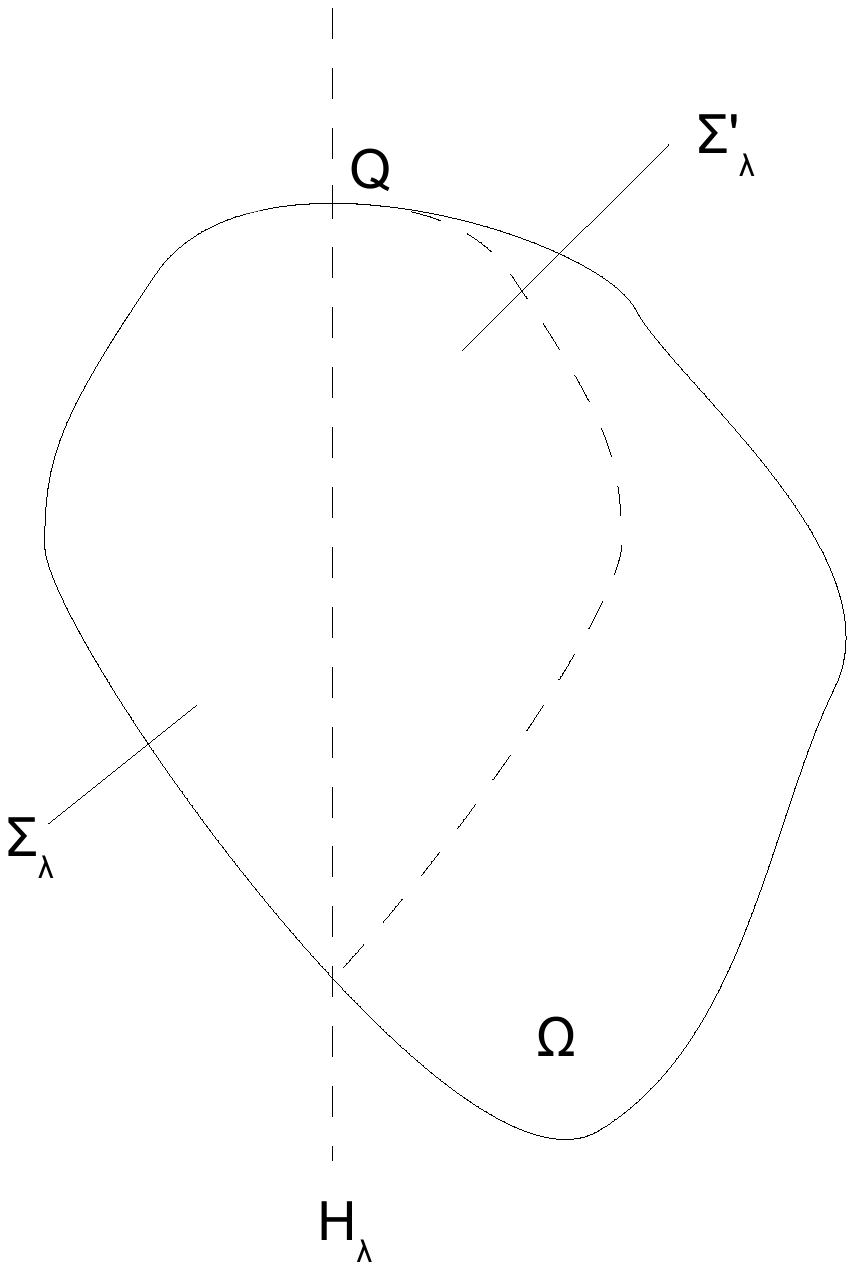}
\end{tabular}
\caption{Configuration 1 (on the left) and configuration 2 (on the right).} 
\label{fig:configurations}
\end{figure}

Let $\La$ be the first value of $\la$ for which configuration 1 or 2 holds. Note that we may have $\Si'_\la\subset \Om$ for some $\la>\La$, but this will be irrelevant in the proof.
The main point in Serrin's proof is to show that one has  $ R_\La u = u $ inside $\Si_\La$, where $u$ is the solution of \eqref{overdetermined2}. This fact  yields easily the symmetry of $\Omega$ with respect to $H_{\La}$. As the direction of $H_{\La}$  was chosen arbitrarily, it will follow that for any direction, there is an axis of symmetry  for $\Omega$ with that direction. This property implies  that $\Omega$ is a disk.   

\medskip
To show the identity $ R_\La u = u $  inside $\Si_\La$, one argues by contradiction through  three main steps. Assume that the equality does not hold.  Then
\begin{itemize}
\item {\em Step 1}.  One shows that the function $w:= u - R_\La u$ is  positive inside $\Si'_\La$. 
\item {\em Step 2}. One obtains an upper bound for $w$ near the tangency point $P$ (configuration 1) or $Q$ (configuration 2). As the normal derivative of $u$ is constant along $\pa \Omega$, one has $w = \pa_n w = 0$ at $P$ or $Q$. For configuration 2, one can  show furthermore that the second derivatives cancel at $Q$: $\pa_{ij} w = 0$ for all $i,j=1,2$. Denoting by $r$ the distance at the tangency point, these properties  imply that 
\begin{equation} \label{DL}
 w = O(r^2) \:  \mbox{ at $P$ (configuration 1), or } \: w = O(r^3)   \:  \mbox{ at $Q$ (configuration 2).}
\end{equation}
\item {\em Step 3}. One shows a lower bound which contradicts \eqref{DL}. For configuration 1, this lower bound is obtained easily.  Indeed, one knows from Step 1 that  $w > 0$  inside $\Si'_\La$. Hence, Hopf's lemma implies 
$\pa_n w  > 0$ at $P$ (where $n$ is the inward normal), which shows that $w$ should grow at least linearly with $r$ inside the domain.   For configuration 2, $\Si'_\La$ is not regular enough at $Q$ to apply Hopf's lemma. However,  one can still  show that 
$$  \pa_s w > 0 \:  \mbox{ or } \: \pa^2_s w > 0 \: \mbox{ at } \: Q  $$
for any direction $\vec{s}$ entering $\Si'_\La$ non-tangentially: see \cite[Lemma 1]{Serrin}. It follows that $w$ should grow at least like $r^2$
 inside the domain. In both cases, we reach the targeted contradiction. We refer to \cite{Serrin} for all details. \end{itemize}
  
 \medskip 
 Our ambition is to transpose this scheme of proof to problem \eqref{overdetermined}. Due to the nonlocal character of $\flap$, it  requires many modifications:   
 \begin{itemize}
 \item In the case of the Laplacian, Step 1 follows from a simple application of the strong maximum principle, as $w \ge 0$ on $\pa \Si'_\La$. In our fractional (and therefore nonlocal) setting, such use of  the maximum principle is impossible: it would require that $w \ge 0$ on the whole $\R^2\setminus  \Si'_\La$, which is not true ($w$ is an odd function). The appropriate treatment of Step 1 will be addressed in paragraph \ref{positivity}. 
 \item In the case of the Laplacian, Steps 2 and 3 rely on the regularity of $u$ up to the boundary. In the case of $\flap$, we only have $C^{0,1/2}$ regularity of $u$, which implies substantial changes. Loosely: 
 \begin{itemize}
 \item  the upper bounds will follow from the asymptotic expansion \eqref{asymptoticexp2} of $u$ near the boundary.
 \item the lower bounds will follow from the construction of refined subsolutions. 
 \end{itemize}
 \end{itemize}

Additionally, we emphasize that the constant $c_0$ in \eqref{overdetermined} is necessarily strictly positive. Indeed, $c_0\geq 0$ by the maximum principle. The fact that $c_0>0$ is a consequence of the Hopf Lemma for the fractional laplacian. In the present context, we may present a self-contained proof: if $c_0=0$, we  jump to paragraph \ref{ssec:config1}. The arguments developed there (for the function $w$ instead of $u$) allow us to conclude that $u\equiv 0$. This is excluded by the equation $(-\Delta)^{1/2} u=1$. Therefore we restrict our analysis to $c_0>0$.
 
 \subsection{Positivity of $w$} \label{positivity}
 We shall first prove that if non-identically $0$, $\check{w} := R_\La u - u$ is positive inside $\Si_\La$ (which amounts to achieving Step 1).
We shall rely on ideas developed by M. Birkner, J. Lopez-Mimbela and A. Wakolbinger in article \cite{BiLoWa}. They show there the radial symmetry of solutions of some semilinear fractional problems  
$$ (-\Delta)^\alpha u  + F(u)  =  0, \quad x \in B(0,1), \quad u = 0 \text{ in } \R^n\setminus B(0,1)$$
set in the unit ball $B(0,1)$ of $\R^n$. Their proof is based on an adaptation of the moving plane method, and its baseline can be used  to show  the positivity of $w$. Nevertheless, several changes are needed, and simplifications of the original arguments can be made, as we now describe.

\medskip
First of all, we introduce the set 
$$ {\cal L} \: := \: \{ \la \in ]0, \La[, \quad  s.t. \quad  \forall 0 <  \gamma \le \la, \quad R_\gamma u - u \ge 0 \:  \mbox{ on } \: \Si_\gamma, \:  \mbox{ and } \: \pa_1 u > 0 \:  \mbox{ on } \:  H_\gamma \cap \Omega \}. $$
The main point is to show that ${\cal L} = ]0,\La[$. We proceed in several steps:
\begin{itemize}
\item {\em Statement 1: $ \: ]0, \eps] \subset {\cal L}$ for $\eps > 0$ small enough}. 
\end{itemize}
To prove that $]0, \eps[ \subset {\cal L}$, it is enough to show that $\pa_1u > 0$ for all $x \in \Omega$ with $0 < x_1 < \eps$. From the condition $\pa_n^{1/2} u = c_0 > 0$, and the expansion  \eqref{asymptoticexp}, we know that 
$$ \pa_r u \: = \: \frac{1}{2} c_0 \, r^{-1/2} \: + \: O(1), \quad   \pa_s u \: = \: O(1) $$
as the distance $r$ to the boundary (measured inside the domain) goes to zero. It follows that 
$$ \pa_1 u \: = \: \pa_1 r \, \pa_r u \: + \: \pa_1 s \, \pa_s u \: = \:    \frac{1}{2} c_0 \,  (\pa_1 r) \, r^{-1/2} + O(1), \quad r \rightarrow 0.$$
Moreover,   for $0 < \la  < \La$, for all inward normal vectors $n$ along $\pa \Omega \cap \{ x_1 < \la \}$, one has $e_1 \cdot n > 0$, where $e_1 = (1,0)$. It follows that  $\pa_1 r$ has a positive lower bound for $0 < x_1 < \eps$ small enough. The statement follows.

\begin{itemize}
\item {\em Statement 2: $\: {\cal L}$ is an open subset of $]0, \La[$. }
\end{itemize}
The point is to show that if $ \la \in {\cal L}$, $[\la, \la + \eps [ \subset  {\cal L}$ for $\eps > 0$ small enough. This statement will be deduced from  the following: 
\begin{lemma} \label{lemma positivity}
Let $\la < \La$. Then

\medskip 
i) For all $x \in \pa \Omega \cap \{ x_1 < \la \}$, $R_\la u(x) - u(x) > 0$. 

\medskip
ii) Assume that $R_\la u - u \ge 0$ on $\Si_\la$. Then, $R_\la u - u > 0$  on $\Si_\la$. In particular, one has $R_\la u - u > 0$  on $\Si_\la$ for any $\la \in {\cal L}$. 
\end{lemma} 
 \begin{proof}We remind that $u > 0$ in $\Omega$ by the maximum principle.
 
\medskip
 i)  As $\la < \La$, for all  $x= (x_1,x_2) \in \pa \Omega \cap \{ x_1 < \la \}$, $x_\la := (2\la - x_1,x_2) \in \Omega$. In particular:
  $$R_\la u(x) - u(x) = u(x_\la) > 0. $$
  
  \medskip
  ii) We assume that $\check{w}:= R_\la u - u \ge 0$ on $\Si_\la$. Note that  on $\{ x_1 < \la\}\setminus \Si_\la$, one has $\check{w} = R_\la u \ge 0$, so that $\check{w} \ge 0$ on the whole half-space $\{ x_1 < \la\}$. Note also that  by i) and the continuity of $u$, the fonction $\check{w}$ is not identically zero in $\Si_\la$. We want to show that $\check{w} > 0$ in $\Si_\la$. We introduce the harmonic extension $\check{W}$ of $\check{w}$, defined on $\R^3_+$. We remind that 
  $$ \Delta  \check{W}= 0 \: \mbox{ on } \:  \R^3_+, \quad \check{W} \vert_{z = 0} = \check{w} \: \mbox{ on } \: \R^2. $$
  Moreover, as $\flap u = 1$ in $\Omega$,  $\pa_z \check{W}\vert_{z = 0}  = 0$ in $\Si_\la$. 
 We can express  $\check{W}$ with  the Poisson kernel: 
  $$ \check{W}(x_1,x_2,z) = \frac{z}{2\pi}\int_{\R^2} \frac{dt_1 dt_2}{((x_1 - t_1)^2 + (x_2 - t_2)^2 + z^2)^{1/2}} \check{w}(t_1,t_2). $$  
  As $R_\la \check{w} = -\check{w}$, this integral formula can be written
\begin{multline*} \check{W}(x_1,x_2,z) = \frac{z}{2\pi}\int_{\{ t_1 < \la \}} dt_1 dt_2 \\
\left( \frac{1}{((x_1 - t_1)^2 + (x_2 - t_2)^2 + z^2)^{1/2}} - \frac{1}{(x_1 + (t_1-2 \la))^2 + (x_2 - t_2)^2 + z^2)^{1/2}} \right)\check{w}(t_1,t_2). 
\end{multline*}
 As $\check{w} \ge 0$   on $\{ t_1 < \la \}$ and not identically zero, we deduce from this formula that $\check{W} > 0$ for $x_1 < \la$, $z > 0$. 
 
 \medskip
 To conclude on the positivity of $\check{w}$, we assume {\it a contrario} that $\check{w}(x^*) = 0$ for some $x^* = (x_1,x_2) \in \Si_\la$. The  function $\check{W}$ is smooth up  to the boundary in the vicinity of $(x_1^*,x_2^*,0)$, and $0= \check{W}(x^*_1,x^*_2,0) < \check{W}(x,z)$ for all $(x,z) \in \R^3_+$.  By Hopf's lemma, we should have $\pa_z \check{W}(x^*_1,x^*_2,0) > 0$ hence reaching a contradiction.  This concludes the proof. 

 \end{proof}

Back to Statement 2: we argue by contradiction, as in \cite[paragraph 4.2]{BiLoWa}. We assume that there is $\la \in {\cal L}$ and a decreasing sequence $(\la_n) \subset ]0,\La[ \setminus {\cal L}$ converging to $\la$. Up to the extraction of a subsequence, we can assume that there exists a sequence $(\gamma_n)$ such that $\lambda<\gamma_n<\lambda_n$ and 
\begin{enumerate}
\item either there exists a sequence $(x_n)$, $x_n \in  \Si_{\gamma_n}$ for all $n$, such that $x_n \rightarrow x^*$ and $u(x_n) > R_{\gamma_n}  u(x_n)$ for all $n$. 
\item or there exists a sequence $(x_n)$, $x_n \in  H_{\gamma_n} \cap \Omega$ for all $n$, such that $x_n \rightarrow x^*$ and $\pa_1 u(x_n) \le 0$ for all $n$. 
\end{enumerate}

\medskip
We first consider case 1: since $\lambda \in \cal L$, we have $R_\la u-u\geq 0$ on $\Sigma_\la$. As $u$ is continuous, we get $u(x^*) - R_{\la}u( x^*) \ge 0$. By  Lemma \ref{lemma positivity}, we cannot have $x^* \in \Si_\la\cup (\pa \Om\cap \{x_1<\lambda\})$. Hence, $x^* \in H_\la \cap \overline{\Omega}$. If $x^* \in H_\la \cap \pa \Omega$, one has 
$ \pa_n^{1/2} u(x^*) = c_0 > 0$. Moreover, as $\la < \La$, $e_1 \cdot n^* > 0$ for $n^*$  the inward normal vector at $x^*$: indeed, since $\la<\La$, configuration 2 has not been met, and therefore $e_1\cdot n(y)$ does not vanish for all $y\in \pa \Om$ such that $y_1\leq \la$. Since $e_1\cdot n(y)=1$ when $y=(0,y_2)\in \pa \Om$, we infer $e_1 \cdot n^* > 0$.

  Reasoning as in the proof of Statement 1, we obtain that $u$ is strictly increasing with $x_1$ in the vicinity of $x^*$ (inside the domain). Hence, for $n$ large enough, we get 
$$ R_{\gamma_n}u( x_n) - u(x_n) > 0$$
in contradiction with the assumption $u(x_n) > R_{\gamma_n} u(x_n)$ for all $n$. If  $x^* \in H_\la \cap \Omega$, then $\pa_1 u(x^*) > 0$ (because $\la \in {\cal L})$. This means again that $u$ is strictly increasing with $x_1$ in the vicinity of $x^*$, which yields the same contradiction as before. 
 
 \medskip
It remains to consider case 2. With the same reasoning as in case 1, we get that $x^* \in H_\la \cap \overline{\Omega}$, and that $u$ is strictly increasing with $x_1$ in the vicinity of $x^*$. This contradicts the assumption $\pa_1 u(x_n) \le 0$ for all $n$.

 \begin{itemize}
\item {\em Statement 3: $\: {\cal L} = ]0, \La[$. }
\end{itemize}
From the previous statements on ${\cal L}$, we know that ${\cal L} = ]0, \La_{max}[$ for some $\La_{max} \le \La$.  Again, we shall argue by contradiction and assume that $\La_{max}< \La$. For all $\la < \La_{max}$, one has $R_{\la} u - u > 0$ in $\Si_\la$, so that by continuity of $u$,  $\: R_{\La_{max}} u - u \ge 0$ in $\Si_{\La_{max}}$.  Lemma \ref{lemma positivity} implies in turn that
$$ R_{\La_{max}} u - u >  0  \quad \mbox{ in } \: \Si_{\La_{max}}. $$
We want to show that $\pa_1 u > 0$ on $H_{\La_{max}} \cap \Omega$. This would imply that $\La_{max} \in {\cal L}$,  so the contradiction. 

\medskip
As $\check{w} := R_{\La_{max}} u - u$ is odd with respect to $H_{\La_{max}}$, it is enough to prove that $\pa_1 \check{w} < 0$ on $H_{\La_{max}}\cap \Om$. We introduce its harmonic extension, as in the proof of Lemma \ref{lemma positivity}: it satisfies  
$$ \Delta \check{W}  = 0 \: \mbox{ on } \:  \R^3_+, \quad \check{W} \vert_{z = 0} = \check{w} \: \mbox{ on } \: \R^2, \quad \pa_z \check{W} \vert_{z = 0}  = 0 \: \mbox{ on } \: 
\Si_{\La_{max}}.$$ 
We extend $\check{W} $ to the lower half-space by setting
$$ \check{W} (x_1,x_2,z) \: := \: \check{W} (x_1,x_2,-z) \quad \mbox{ for } \: z < 0. $$
As $\pa_z \check{W} \vert_{z = 0}  = 0$ on $\Si_{\La_{max}}\times \{0\}$, this extension is harmonic through $\Si_{\La_{max}}$, that is outside $(\R^2\setminus\Si_{\La_{max}}) \times \{0\}$. 
Notice also that as in the proof of Lemma \ref{lemma positivity}, $\check{w}\geq 0$ on $\{x_1<\La_{max}\}$, and therefore $\check{W}\geq 0$ on $\{x_1<\La_{max}\}$.
Let $x^* \in  H_{\La_{max}} \cap \Omega$. $\check{W} $ is smooth and harmonic near $(x^*_1, x^*_2, 0)$. Moreover, one has 
$$ 0  = \check{w}(x^*) = \check{W} (x^*_1, x^*_2, 0) < \check{W} (x,z), \quad \mbox{ say for all } \:  (x,z) \in {\cal O} := \Si_{\La_{max}} \times \R $$ 
Applying Hopf's lemma in ${\cal O}$  (for which $(x^*_1, x^*_2, 0)$ is a boundary point), we obtain that 
$ \pa_1 \check{W} (x^*_1, x^*_2, 0) < 0$
that is $\pa_1  \check{w}(x^*) < 0$, as expected. 

 \bigskip
 As a conclusion of our analysis, we obtain that $ R_{\la} u - u > 0$ on $\Si_\la$ for all $\la < \La$.  By continuity of $u$, we obtain that 
 $  \check{w}:= R_{\La} u - u \ge 0$ on $\Si_{\La}$. From there, there are two possibilities: 
 \begin{itemize}
 \item either  $ \check{w}=0$  on $\Si_{\La}$. It follows easily that $\Omega$ is symmetric with respect to $H_\La$. As explained in paragraph \ref{reminders}, the direction of $H_\La$ being arbitrary, it follows that $\Omega$ is a disk. 
 \item or $ \check{w}(x) > 0$ for some point $x$ in $\Si_\La$.  But then, following exactly the proof of point ii) in Lemma \ref{lemma positivity}, we obtain that  $ \check{w}  > 0$ on $\Si_\La$ (or $w := u - R_\La u > 0$ on $\Si'_\La$).  
 \end{itemize}
 The rest of the section aims at excluding the second possibility. {\em From now on, we  assume that  $w  > 0$ on $\Si'_\La$, and will establish contradictory lower and upper bounds on $w$.} 
 
 \medskip
 First of all, we state a lemma, to be used in both configurations 1 and 2. 
\begin{lemma} \label{lemmex1x3}
Let $W$ be the harmonic extension of $w$ to $\R^3_+$. For all $R > 0$ sufficiently large there exists $c_R > 0$ such that 
$$ W(x,z) \: \ge \:  c_R  \, (x_1 - \La) \, z, \quad \mbox{ for all } \: \La \le x_1 \le R, \quad   -R \le x_2 \le R,  \quad   0 \le z \le R.  $$
 \end{lemma}
\begin{proof}
 Let $R > 0$ large,  and  $\displaystyle Q_R :=  ]\La, R[ \times ]-R,R[ \times  ]0,R[$. As $W$ and $W' :=  x \mapsto (x_1 - \La) \, z$ are both harmonic functions on $Q_R$, we can use the maximum principle to compare them:  we just have  to show that $W \ge c_R W'$ on $\pa Q_R$ for $c_R > 0$. We take $R$ large enough so that the vertical sides of the cube (other than the one supported by  $\{x_1 = \La\}$) do not intersect $\overline{\Omega} \times \{0\}$.  Then, we  distinguish between the different sides: 
 \begin{itemize}
 \item On $\pa Q_R \cap  \{z = 0\}$ or   $\pa Q_R  \cap \{ x_1 = \La\}$,  one has $W \ge 0$ (remind that $w \ge 0$ on $\{x_1 \ge \La\}$), and $W' = 0$. The inequality is clear.   
 \item Let $\Gamma_{1,R} := \pa Q_R \cap \{x_1 =  R\}$. Let $\eps > 0$. As  $W$ is continous and $> 0$ on the compact set  $\Gamma_{1,R}\cap \{ z \ge \eps \}$,  one can find $c = c_\eps > 0$ such that $W \, \ge \,  c \, W'$. To  have the same inequality on the whole $\Gamma_{1,R}$,  it is enough to show that 
\begin{equation} \label{liminfgamma1}
 \liminf_{\substack{(x,z) \in \Gamma_{1,R}, \\ z \rightarrow 0^+}} \frac{W(x,z)}{W'(x,z)}  = \frac{1}{R-\La} \liminf_{\substack{(x,z) \in \Gamma_{1,R}, \\ z \rightarrow 0^+}} \frac{W(x,z)}{z} > 0. 
\end{equation}
This is a consequence  of  Hopf's lemma: $W$ is positive harmonic  on $Q_R $, satisfies $W=0$ on $\Gamma_{1,R} \cap \{ z = 0\}$. It follows that
$ \pa_z W > 0$ on $\Gamma_{1,R} \cap \{ z = 0\}$. Note that $\pa_z W$ exists thanks to our choice of $R$, away from $\Omega$.  It is furthermore continuous, so that  it is bounded from below by a positive constant. Inequality \eqref{liminfgamma1} follows. 
\item On $\Gamma_{3,R} := \pa Q_R \cap \{z =  R\}$,  one can proceed exactly as in the case of $\Gamma_{1,R}$: use the positivity of $W$ away from $\{x_1 = \La\}$, and use Hopf's lemma to get $\pa_1 W \ge c > 0$ at $\Gamma_{3,R} \cap \{x_1 =  \La \}$. 
\item We now turn to $\Gamma_{2,R} = \pa Q_R \cap \{x_2 = \pm R \}$. Away from the edge $x_1 = \La, z = 0$, we can as before use the positivity of $W$ and  Hopf's lemma to obtain $W \ge c W'$. It then remains to handle the vicinity of $X = (\La, \pm R, 0)$. The main point is to show that 
\begin{equation} \label{liminfgamma2}
 \liminf_{\substack{(x,z) \in \Gamma_{2,R} \\ (x,z) \rightarrow  X}} \frac{W(x,z)}{(x_1 - \La) \, z} > 0.    
\end{equation}
We can not apply Hopf's lemma: $W$ is harmonic in the dihedra $\{ x_1 > \La, z >  0 \}$, but it does not satisfy the interior sphere condition at $X$. 
As $W(x,z) = 0$ both for $x_1 = \La$ and $z = 0$,  one has 
 \begin{equation} \label{d3wd1w}
 \left\{
 \begin{aligned}
 \pa_z^k W(x,z)  & = 0 \quad \mbox{ for } x_1 = \La, \: \forall k \in \N \\
  \pa_1^k W(x,z) & = 0 \quad \mbox{ for } z = 0,  \: \forall k \in \N.
 \end{aligned}
 \right.
 \end{equation}
($(x,z)$ in the vicinity of $X$).  Hence, to prove \eqref{liminfgamma2}, it is enough to show that 
 $$ \pa_1 \pa_z W(x,z) \ge c > 0, \quad (x,z) \in B(X, \delta) \cap \{ x_1 \ge \La, z \ge 0 \}, $$
 for small enough $\delta$. By continuity of $\pa_1 \pa_z W$, it is enough to show that 
   $ \pa_1 \pa_z W(X) > 0$.
 This positivity condition follows straightforwardly from a famous lemma of Serrin, see \cite[Lemma 1, p308]{Serrin}. In our context, it reads: {\em for any vector $s$ entering the region $\{x_1 > \La, \,  z > 0\}$,   we have $\pa_s W(X) > 0$ or $\pa^2_s W(X) > 0$.}  Taking $s = e_1 + e_3$, we obtain that 
 $$ (\pa_1 + \pa_z) W(X) > 0, \quad \mbox{ or } \:  (\pa_1 + \pa_z)^2 W(X) > 0. $$
 Combining this statement with \eqref{d3wd1w}, we see that the first condition is not realized, and that the second one amounts to $\pa_1\pa_z W(X) > 0$ as expected. 
\end{itemize}
 \end{proof}

 \subsection{Contradictory bounds on $w$: configuration 1}
\label{ssec:config1}

\subsubsection*{Lower bound}
  We start with configuration 1, that is when $\Si'_\La$ is internally tangent to $\pa \Omega$ at a point $P$ not in $H_{\La}$. We introduce a small open disk $D_R \subset \Omega$  of radius $R$, also tangent to $\pa \Omega$ at $P$. We take $R$ small enough so that $\pa D_R \cap \pa \Omega = \{ P \}$. Let $x_R$ be the center of $D_R$, and let $\phi \in [0,2\pi]$ parametrizing $\pa D_R$.    Finally, for any $x$ in $D_R$,  denote by $\rho = \rho(x)$ the distance between $x$ and  the circle $\pa D_R$. One has  
  $x = x_R + ((R-\rho) \cos \phi, (R -\rho) \sin \phi)$. The aim of this paragraph is to prove the following 
  \begin{proposition}   \label{lowerboundcase1}
  Let $\beta > 1/2$. There exists $c_\beta > 0$, $\rho_\beta > 0$ such that for all $x \in D_R$  with $0 < \rho(x) < \rho_\beta$, 
  $$ w(x) \ge c_\beta \,  \rho^\beta. $$ 
  \end{proposition}

 First, we shall  extend the 2d coordinate system $(\rho,\phi)$  to a  3d coordinate system $(\rho,\phi,\theta)$ in the $R/2$-neighborhood of $C_R := \pa D_R \times \{ 0 \}$ (see Figure \ref{fig:coord-sphe}). Any  $(x,z) = (x_1, x_2, z)$ in this neighborhood has a unique projection $p(x,z)$ on $C_R$, which reads: 
 $$ p(x,z) = (x_R,0) +  (R\cos \phi,  R \sin \phi, 0). $$
 We introduce  $\rho = | (x,z) - p(x,z)|$ the distance between $(x,z)$ and $C_R$,  and $\theta  \in [-\pi, \pi[$ the oriented angle between the vector $(x_R, 0) - p(x,z)$ and $(x,z) - p(x,z)$. Then, one can write
\begin{align*} 
 (x,z) & = p(x,z) + (-\rho\cos\theta \cos \phi, -\rho\sin \theta \sin \phi, \rho \sin \theta) \\
 & =  (x_R,0) +  ((R-\rho\cos\theta) \cos \phi, (R-\rho\sin \theta) \sin \phi, \rho \sin \theta)  
 \end{align*}
  The triplet $(\rho, \theta, \phi)$ defines a system of orthogonal curvilinear coordinates in the $R/2$ neighborhood of $C_R$. For $\theta = 0$, it matches the 2d coordinates $(\rho,\phi)$ introduced above.
\begin{figure}
\includegraphics[height=10cm]{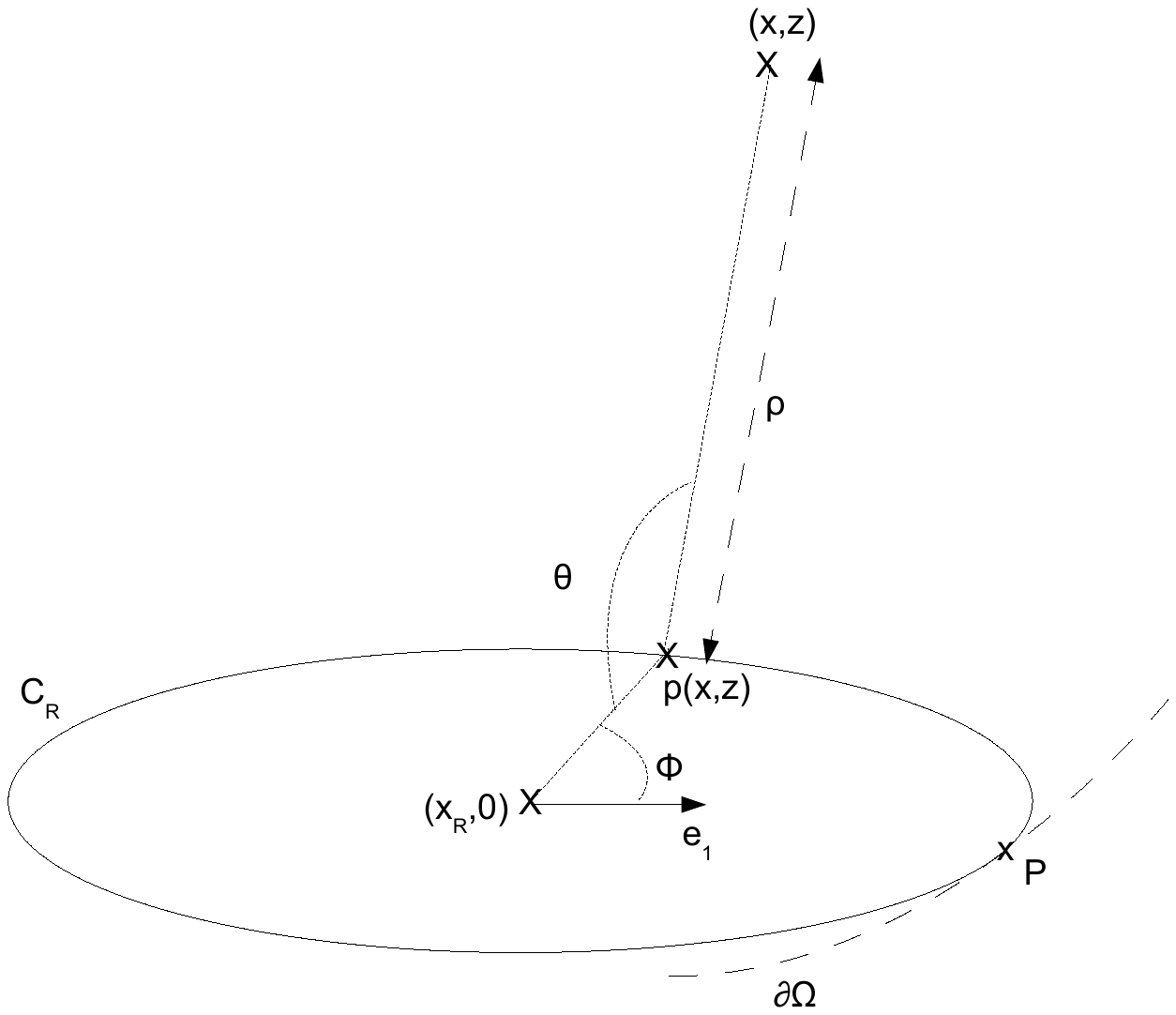}
\caption{The coordinates $(\rho, \theta, \phi)$.}\label{fig:coord-sphe} 
\end{figure}

\medskip
To prove Proposition  \ref{lowerboundcase1}, we shall rely again on the harmonic extension $W$ of $w$.  It satisfies 
$$ \Delta W  = 0 \: \mbox{ on } \:  \R^3_+, \quad W \vert_{z = 0} = w \: \mbox{ on } \: \R^2, \quad \pa_z W \vert_{z = 0}  = 0 \: \mbox{ on } \:  \Si'_{\La}.$$
We can extend $W$ to $\R^3$ into an even function of $z$ (still denoted $W$). Thanks to the last condition, it is harmonic through $\Si'_\La$.   We will show that for $\beta > \frac{1}{2}$,
\begin{equation} \label{lowerbound3dcase1}
W(x,z) \ge c_\beta \,  \rho^\beta \cos(\theta/2), \quad \forall   \rho=\rho(x) < \delta \: \mbox{ small enough. }   
\end{equation} 
This will yield Proposition \ref{lowerboundcase1} for $\theta= 0$. This lower bound on $W$ will follow from the maximum principle. We introduce  $\: U_\delta \, := \, \{ x, \: \rho(x) < \delta \} \setminus \{ \theta = \pi\}$. The open set $U_\delta$ is obtained by removing from the $\delta$-neighborhood of $C_R$ the part of the plane $\{z = 0 \}$ outside $D_R$. 
 Let $W_\beta$ given in $(\rho,\theta,\phi)$ coordinates by $W_\beta(x,z) = \rho^\beta \cos(\theta/2)$. We will show that for $\delta > $0 small enough, one has: 
\begin{description}
\item[i)]  $\Delta W_\beta \ge 0 \:  (= \Delta W) \: $ in  $ \:  U_\delta$. 
\item[ii)] $W \ge c_\beta W_\beta \: $  on  $\: \pa U_\delta$, for some $c_\beta > 0$.
\end{description}
The lower bound \eqref{lowerbound3dcase1} will follow.  

\medskip
{\em Proof of i)}. One can check that the Laplacian in coordinates $(\rho,\theta,\phi)$ reads 
\begin{multline} \label{Laplacian} \Delta  f  = \\
 \frac{1}{\rho(R-\rho\cos \theta)}\left(  \frac{\pa}{\pa \rho} \left( \rho(R-\rho\cos \theta) \frac{\pa}{\pa\rho} f \right) + \frac{\pa}{\pa \theta} \left( \frac{R-\rho\cos\theta}{\rho} \frac{\pa}{\pa \theta} f \right)  + \frac{\pa}{\pa \phi} \left( \frac{\rho}{R-\rho\cos\theta} \frac{\pa}{\pa \phi} f \right) \right). 
 \end{multline}
This implies that 
\begin{align*} 
\Delta  W_\beta & =  \frac{1}{\rho(R-\rho\cos \theta)} \Bigl( \Bigl( (\beta^2-1/4)  \rho^{\beta-1} (R-\rho\cos\theta) - \beta \rho^\beta  \cos \theta
     \Bigr) \cos(\theta/2)  - \frac{1}{2} \rho^\beta  \sin \theta  \sin(\theta/2) \Bigr) \\
 & = \frac{1}{\rho(R-\rho\cos \theta)}  \left( (\beta^2 - 1/4) \rho^{\beta-1} (R- \rho\cos \theta)   -  \beta\rho^\beta  \cos \theta -    \rho^\beta  \sin^2(\theta/2) \right)    \cos(\theta/2)  \\
 & =      \frac{1}{\rho(R-\rho\cos \theta)} \left(  (\beta^2 - 1/4) \, R \,  \rho^{\beta-1} + O(\rho^\beta) \right)   \cos(\theta/2) \ge 0 
\end{align*}
for $\delta$ (and $\rho < \delta$) small enough. 
Note that we used the identity $\sin \theta = 2 \sin(\theta/2) \cos(\theta/2)$ in the second line. 

\medskip
{\em Proof of ii)}. We must distinguish between different zones. 
\begin{itemize}
\item On $\pa U_\delta \cap \{ \theta = \pm \pi \}$, one has $W \ge 0,  W_\beta = 0$, so that the inequality is clear. 
\item The compact set  $\pa U_\delta \cap \{ \theta = 0 \}$ is included in $\Si'_\La$,  so that $W > 0$ there, and as $W$ is continuous,  it is even bounded from below by a positive constant. In particular, one can find $c >0$ such that $W \ge c W_\beta$. 
\item For any $\eps > 0$, we also know that $W > 0$ on  the compact set $\pa U_\delta \cap  \{ |z| \ge \eps \}$. We obtain again that   
$W \ge c W_\beta$ for some $c=c_\eps > 0$. 
\item It remains to show that for $\eps > 0$ small enough, one has 
$$W \ge c W_\beta \:  \mbox{ on } \:  \pa U_\delta \cap  \{ 0 < |z| \le \eps \}.$$
Due to the fact that $W$ is non-negative in $\R^3$ and even in the variable $z$, it is enough to show that 
$$ \liminf_{\substack{(x,z) \in \pa U_\delta,\\ z \rightarrow 0^+}} \: \left| W(x,z) / W_\beta(x,z) \right| > 0. $$
To show such a property, we argue by contradiction. Assume that we can find a sequence $(x^k,z^k)$ satisfying: 
$$(x^k,z^k) \subset \R^3_+ \cap \pa U_\delta,  \quad z^k \rightarrow 0^+, \quad \left| W(x^k,z^k) / W_\beta(x^k,z^k) \right| \rightarrow  0.$$  
Up to the extraction of a  subsequence, we can assume that $x^k,z^k \rightarrow X$ for some $X \in \pa U_\delta $ with $X =(X_1,X_2, 0)$. There are two possibilities: 
\begin{itemize}
\item either $W(X)  > 0$. As $W_\beta$ is bounded, this yields easily a contradiction.  
\item or $W(X) = 0$. In particular, $X$ is outside $C_R$, which implies in turn $W_\beta(X) = 0$. More generally, $W(x_1,x_2,0) = 0$ for all 
$(x_1,x_2)$  close to $(X_1,X_2)$.    Moreover, as $W_\beta$ is smooth near $X$, we obtain  
\begin{equation*} |W_\beta(x,z)| = O(z), \quad \mbox{  as }  z \rightarrow 0^+, \quad x \: \mbox{ in the vicinity of  } X.
\end{equation*}
From there, we get that 
\begin{equation} \label{limsup}
 \limsup_{k \rightarrow \infty} |W(x^k,z^k)|/z^k = 0
\end{equation}
As  $x_1^k$ remains away from $\La$, this limit contradicts Lemma \ref{lemmex1x3}. 
\end{itemize}
\end{itemize}

\subsubsection*{Upper bound}
We now turn to an upper  bound for $w$.  Let $x \in \Si'_\La$. For $x$ close  to $P$, we can write in a unique way
$$ x \: = \:  x^* \, + \,  r  \, n, \quad \mbox{ resp. } x \: = \:  \bar{x}^* \, +  \, \bar{r} \,  \bar{n}, \quad r, \bar{r} > 0,    $$
  where $x^* \in \pa \Omega$, resp. $\bar{x}^* \in  \pa \Si'_\La$, and $n = n(x^*)$, resp. $\bar{n} = \bar{n}(x^*)$,  refers to the inward normal to $\pa \Omega$ at $x$ , resp. to the inward normal to $\pa   \Si'_\La$ at $\bar{x}$. 
  We now use the results of Costabel et al. \cite{CoDauDu}, and the fact that $\pa^{1/2}_n u = c_0$ at the boundary. Following \eqref{asymptoticexp2},  we get the expansions
$$ u (x) = c_0 r^{1/2} + O(r), \quad R_\La u(x) = c_0 \bar{r}^{1/2} + O(\bar{r}), \quad r,\bar{r} \rightarrow 0. $$  
Now, {\em we remark that the inward normal vectors $n$ and $\bar{n}$  coincide at $P$}. The expansions above then lead to  the upper bound
\begin{equation} \label{upperbound}
  w(P \, + \, r \, n ) = O(r), \quad r > 0 \: \mbox{ small enough}.  
\end{equation}
\medskip
Furthermore, the coordinate $r$ along the normal at $P$ coincides with the  coordinate $\rho$ from Proposition \ref{lowerboundcase1}. Applying this proposition with $\beta = 3/4$, we get
$$ w(P \, + \, r \, n)  \, \ge \, c \, r^{3/4}, \quad  r > 0 \: \mbox{ small enough} $$ 
Comparison between this lower bound and the upper bound \eqref{upperbound} gives a contradiction.

  \subsection{Contradictory bounds on $w$: configuration 2}
In this paragraph, we investigate configuration 2, in which $H_\La$ is orthogonal to $\pa \Omega$ at some point $Q$. 
\subsubsection*{Lower bound} 
We shall prove the following
\begin{proposition} \label{lbcase2}
Let $\vec{\nu} = (1,\nu)$ be a vector entering $\Si'_\La$ at $Q$. Let $\beta > 1/2$. There exists $c_\beta > 0$ such that 
$$ w(Q + t \vec{\nu}) \, \ge \, c_\beta \, t^{1 + \beta}, \quad \mbox{ for $t$ small enough}. $$
\end{proposition}
To prove this proposition, we consider again a small disk $D_R$, of radius $R$, tangent to $\pa \Omega$ at $Q$.  Exactly as in the proof of Proposition  
 \ref{lowerboundcase1}, we introduce the 3d coordinates $(\rho,\theta,\phi)$ in a $R/2$-neighborhood of $C_R = \pa D_R \times \{0\}$. We also introduce the 3d harmonic extension $W$ of $w$, even with respect to the variable $z$. Our aim is to prove a lower bound on $W$ which implies Proposition \ref{lbcase2}, namely
 \begin{equation} \label{lowerboundcase2}
  W(x,z) \ge c_\beta \, \rho^\beta \cos(\theta/2) \cos\phi, \quad \forall -\pi/2 \le \phi \le \pi/2, \quad \forall \theta \in [-\pi,\pi], \quad \forall \rho = \rho(x) < \rho_\beta 
  \end{equation}
with $\rho_\beta>0$ small enough. Therefore,  let $W_\beta(x,z) =  \rho^\beta \cos(\theta/2) \cos\phi$. Comparison between  $W$ and $W_\beta$ will come from  the maximum principle. 
  The main change with respect to configuration 1 is that $W = 0$ on the hyperplane $\{x=(x_1,x_2,z),  \: x_1 = \La\}$ (corresponding to $\phi = \pm \pi/2$). We have to restrict to the open set 
  $$ U^+_\delta = U_\delta \cap \{ x_1 > \La\} =  \{  \rho < \delta, \: x_1 > \La, \: \theta \neq \pm\pi \}. $$
 We shall prove that 
 \begin{description}
 \item[a)]  $\Delta W_\beta \ge 0  \: ( = \Delta W) \: \mbox{ on } U^+_\delta$.  
\item[b)]  $ W \ge c_\beta W_\beta$ on $\pa U^+_\delta$, $\delta > 0$ small. 
\end{description} 

\medskip
{\em Proof of a).} We use again formula \eqref{Laplacian}. It gives
\begin{align*}
 \Delta W_\beta & = \frac{1}{\rho (R - \rho\cos \theta)} \biggl(  \left(  (\beta^2 - 1/4) \, R \,  \rho^{\beta-1} + O(\rho^\beta) \right)   \cos(\theta/2) \cos \phi \\
 & \quad - \frac{\rho^{\beta+1}}{R - \rho \cos \theta} \cos(\theta/2) \cos\phi \biggr)\\
 & =  \frac{1}{\rho (R - \rho\cos \theta)} \left(  (\beta^2 - 1/4) \, R \,  \rho^{\beta-1} + O(\rho^\beta) \right)   \cos(\theta/2) \cos \phi \ge 0.
 \end{align*}

\medskip
{\em Proof of b).} 
Let $\eps > 0$. Away from $H_\La$, that is over $\pa U^+_\delta \cap \{ x_1 > \La + \eps\}$, we can perform exactly the same analysis as in the proof of Proposition \ref{lowerboundcase1}, point ii). Thus, 
$$ W \, \ge \,  c \,  W_\beta, \quad \mbox{ over } \:   \pa U^+_\delta \cap \{ x_1 > \La + \eps\} \: \mbox{ for some } \:  c=c_\eps > 0.$$

\medskip
It remains to treat the vicinity of $H_\La$ inside $\pa U^+_\delta$. For  $x_1 = \La$ (that is $\phi = \pm \pi/2$), one has $W(x,z) = W_\beta(x,z) = 0$. Also,
for  $\theta =  \pm\pi$, one has $W(x,z) \ge 0,  W_\beta(x,z) = 0$.  Everywhere else, $W_\beta(x,z) > 0$ and $W(x,z) > 0$. Hence,  it remains to show that 
 $$ \liminf_{\substack{(x,z) \in \pa U^+_\delta\setminus \{|\theta| = \pi\},  \\ x_1 \rightarrow \Lambda^+}} \: \left| W(x,z) / W_\beta(x,z) \right| > 0. $$
Again, we argue by contradiction: we assume  that there is a sequence 
$$(x^k,z^k) \subset  \pa U^+_\delta\setminus \{\theta = \pi\}, \quad x^k_1 \rightarrow \La^+, \quad \left| W(x^k,z^k) / W_\beta(x^k,z^k) \right| \rightarrow  0.$$  
Up to a subsequence, we can  assume that  $(x^k,z^k) \rightarrow X$ for some $ X = (\La,X_2,Z) \in \pa U^+_\delta$. As $W$ and $W_\beta$ are even in $z$, we can also assume that $z^k \ge 0$ for all $k$.  
There are several cases:
 \begin{itemize}
 \item $(X_1,X_2) = (\La, X_2) \in \Omega$. $W$ is harmonic and positive on $\{(x, z), | (x,z) - X | \le \eps, x_1 > \La\}$, $\eps > 0$ small enough. Moreover, for all $(x,z)$  in this set, we have $W(x,z) > 0 = W(\La,x_2,z)$. We can apply Hopf's lemma, which yields $\pa_1 W(\La,x_2,z) \ge C > 0$.  Hence, 
 $$ \liminf_{k \rightarrow \infty} \frac{W(x^k,z^k)}{|x^k_1 - \La|}  > 0 $$
 which yields in turn 
 $$  \liminf_{k \rightarrow \infty} \frac{W(x^k,z^k)}{|W_\beta(x^k,z^k)|}  > 0. $$
 Indeed, $|W_\beta(x^k,z^k)| = |W_\beta(x^k,z^k) - W_\beta(\La, x_2^k,z^k)| = O(|x^k_1 - \La|)$ due to the regularity of $W_\beta$ near $X$. Thus, we reach a contradiction. 
 \item  $(X_1,X_2) = (\La, X_2) \notin \Omega, \: Z  \neq 0$.  We are still in a situation where we can apply Hopf's lemma to $W$ near $X$, which yields the same contradiction as above. 
\item  $(X_1,X_2) = (\La, X_2) \notin \Omega, \:  Z = 0$. Although $W$ is harmonic and positive on the dihedra $\{ x_1 > \La, z > 0 \}$, we can not apply Hopf's lemma because the dihedra does not satisfy  the interior sphere condition at $X$.  Nevertheless, thanks to Lemma \ref{lemmex1x3}, we get that for some $c' > 0$, for all $k$ large enough
$$ W(x^k,z^k) \: \ge \:  c \, (x^k_1 - \Lambda) z^k \: \ge \:  c' \, \left(\frac{\pi}{2} - |\phi^k|\right) \, (\pi - \theta^k) \: \ge \: c'' \, W_\beta(x_k,z^k). $$  
This ends the proof of \eqref{lowerboundcase2}, and so the proof of Proposition \ref{lbcase2}. 
\end{itemize}

\subsubsection*{Upper bound}
We now look for an upper bound for $w$ along the ray $Q \, + \,  t \, \vec{\nu}$, $\vec{\nu} \, = \,  (1, \nu)$, $t > 0$ small. We show 
\begin{equation} \label{ubcase2}
w(Q + t \vec{\nu}) = O(t^2), \quad \mbox{ as $t$ goes to } 0^+. 
\end{equation}
Of course, such an upper bound leads to a contradiction with  Proposition \ref{lbcase2} (for $1/2 < \beta < 1$). 

\medskip
As in  configuration 1, this bound follows from the behaviour of $u$ near $\pa \Omega$, as given by \eqref{asymptoticexp2}. We remind that for $x$ near $Q$, 
$$ u(x) \: = \: c_0 \, r^{1/2} \: + \: C_1(x_1) \, r^{3/2} \: + \: u_{reg}(x) $$
where 
\begin{itemize}
\item $r = r(x)$ is a normal coordinate in the vicinity of $\pa \Omega$ (which corresponds to $r=0$). 
\item $c_0$ is the constant fractional  normal derivative at $\pa \Omega$,  $\: C_1$ is a smooth function of $x_1$. 
\item $u_{reg}$ is $C^2$ near $\pa \Omega$, and vanishes at $\pa \Omega$. 
\end{itemize}
We can then write 
$$ w(x) = c_0 w_{1/2}(x) + C_1 w_{3/2}(x) + w_{reg}(x) + O(|x_1-\La| r^{3/2}) $$
with $C_1 := C_1(\La)$, 
\begin{equation*}
w_{1/2}(x) \: := \: r^{1/2}(x) - r^{1/2}(\check{x}), \quad w_{3/2}(x) \: := \: r^{3/2}(x) - r^{3/2}(\check{x}), 
 w_{reg}(x) \: := \: u_{reg}(x) - u_{reg}(\check{x}), 
\end{equation*}
where $\check{x} := (2 \La - x_1, x_2)$. Note that $w_{1/2}$, $w_{3/2}$, $w_{reg}$ vanish at $Q$.

\medskip
Let us first prove that $\na w_{reg}(Q) = 0$. The idea is the same as in the paper of Serrin \cite{Serrin}. Clearly, $w_{reg}(\La,\cdot) = 0$, so that $\pa_2 w_{reg}(\La, \cdot) = 0$. In particular, $\pa_2 w_{reg}(Q) = 0$. We then have to show that $\pa_1 w_{reg}(Q)$. Near $Q$, we can write $\pa \Omega$ as a graph: $x_2 = \psi(x_1)$, with $\psi'(\La) =0$. Moreover, 
$$ u_{reg}(x_1, \psi(x_1)) = 0 \: \Rightarrow \: \pa_1 u_{reg} + \psi' \pa_2 u_{reg} = 0, $$
so that $\pa_1 u_{reg}(\La,0) = 0$, and from there,  $\pa_1 w_{reg}(Q) = 2 \, \pa_1 u_{reg}(Q) =0$. It follows that $w_{reg}(Q \,  + \,  t \, \vec{\nu}) = O(t^2)$ for $t \rightarrow 0^+$. 

\medskip
We still  have to control $w_\alpha$, $\alpha = 1/2$ or $3/2$. For $x$ close enough to $\pa \Omega$, we denote $p=p(x)$ its orthogonal projection on $\pa \Omega$. Near $Q$, it reads $p = (p_1, \psi(p_1))$, with 
\begin{equation} \label{ortho}  \begin{pmatrix} x_1 - p_1 \\ x_2 - \psi(p_1) \end{pmatrix} \cdot   \begin{pmatrix} 1\\ \psi'(p_1) \end{pmatrix} \: = \: 0. 
\end{equation}
Then, $r(x)^2 = (x_1 - p_1)^2 + (x_2 - \psi(p_1))^2 =  (1 + \psi'(p_1)^2) (x_2 - \psi(p_1))^2, $
and 
$$w _\alpha(x) = (1 + \psi'(p_1)^2)^{\alpha/2} |x_2 - \psi(p_1)|^\alpha  - (1 + \psi'(\check{p}_1)^2)^{\alpha/2} |x_2 - \psi(\check{p}_1)|^\alpha $$ 
with $\check{p} := p(\check{x})$. 
One may then Taylor expand $w_\alpha(Q + t \vec{\nu})$ with respect to $t$ ($x_1 = \La + t$, $x_2 = \psi(\La) + \nu t$). Indeed,  relation \eqref{ortho}:
  $$
\begin{aligned}
 \La+ t -p_1 + (\psi(\La)+\nu t-\psi(p_1))\psi'(p_1)=0,\\
 \mbox{ resp. } 2 \La- (\La+t)  - \check{p}_1 + (\psi(\La)+\nu t-\psi(\check{p}_1))\psi'(\check{p}_1)= 0
\end{aligned}
$$
 leads to $p_1 = \La + t + O(t^2)$, resp.  $\check{p}_1 = \La - t + O(t^2)$. Then,  since $\psi'(\La)=0$, 
 $$
 \psi(p_1)= \psi(\La) + \frac{\psi''(\La)}{2} t^2 + O(t^3),\quad \psi(\check{p}_1)= \psi(\La) + \frac{\psi''(\La)}{2} t^2 + O(t^3),
 $$
 so that
 $$ 
 \begin{aligned}
  |x_2 - \psi(p_1)|^\alpha = |\nu t|^\alpha \left(1-\frac{\alpha \psi''(\La)}{2\nu} t+ O(t^2)\right), \\ \mbox{ resp. }  \:   |x_2 - \psi(\check{p}_1)|^\alpha  =  |\nu t|^\alpha  \left(1-\frac{\alpha \psi''(\La)}{2\nu} t+ O(t^2)\right) 
 \end{aligned}
$$
and 
$$   (1 + \psi'(p_1)^2)^{\alpha/2} = 1 + \frac{\alpha}{2} \psi''(\La) t^2 + O(t^3), \: \quad    \mbox{ resp. }  \:  (1 + \psi'(\check{p}_1)^2)^{\alpha/2} =  1 + \frac{\alpha}{2} \psi''(\La) t^2 + O(t^3). $$
Combining previous Taylor expansions yields: $ w_\alpha(Q + t \vec{\nu}) = O(t^{2+\alpha})$ as $t \rightarrow 0^+$. Together with the estimate $w_{reg}(Q + t \vec{\nu}) = O(t^2)$, it implies the upper bound   \eqref{ubcase2}.

\begin{remark} If the domain $\Om$ is not connected, our method can still be applied, provided the connected components of $\Om$ can be ``ordered'', in the following sense: for every direction of the moving hyperplanes $H_\la$, for all $\la<\La$, $H_\la$ intersects at most one connected component of $\Om$. In this case, it can be checked that our arguments remain valid. In particular, if $\Om$ has radial symmetry, we infer that $\Om$ is a disc.

\label{rmk:rings}
\end{remark}

\section*{Appendix}

%
%
%
%
%
%
%
%

This appendix is devoted to the calculation of the jacobian of the change of variables $x=(x_1,x_2)\to (s,r)$, where $x$ belongs to a tubular neighbourhood $\Td$ of $\pa \Om$, $s$ is the arc-length, and $r=d(x,\pa \Om)(2\mathbf{1_{x\in \Om}} - 1)$. Using the same notation as in Section \ref{sec:derivee-forme}, we write
$$
x=p(s) - r n(s),
$$
where $p(s)\in \R^2$ is the point of $\pa \Om$ with arc-length $s$, and $n(s)$ is the outward pointing normal.
We deduce that
$$\begin{aligned}
\frac{\pa x}{\pa s}= \tau (s)- r \frac{dn}{ds}=(1+\kappa(s) r) \tau (s),\\
\frac{\pa x}{\pa r}=-n(s).
\end{aligned}
$$
Let $\theta$ be the oriented angle between $e_1$ and $n$. Then, up to a change of the orientation of the curve,
$$
n(s)= \begin{pmatrix}
\cos \theta(s)\\\sin \theta(s)       
      \end{pmatrix},\quad
\tau (s)= \begin{pmatrix}
-\sin \theta(s)\\\cos \theta(s)       
      \end{pmatrix}.
$$
We infer that the jacobian matrix of the change of variables is
$$
\begin{pmatrix}
\ds\frac{dx_1}{ds}& \ds\frac{dx_1}{dr}\\ &\\
\ds\frac{dx_2}{ds}& \ds\frac{dx_2}{dr}
\end{pmatrix}
= 
\begin{pmatrix}
-(1+\kappa(s) r) \sin \theta(s)& - \cos \theta(s) \\
(1+ \kappa(s) r) \cos \theta(s)&-\sin \theta(s)
\end{pmatrix},
$$
and therefore the jacobian of the change of variables is $|1+\kappa(s) r|$.

\section*{Acknowledgements}
The authors wish to thank Jean-Michel Roquejoffre and Cyril Imbert for helpful discussions, and they acknowledge   grant ANR-08-JC-JC-0104.

\bibliography{lapla_frac}

\end{document}